\documentclass[12pt,reqno, a4paper]{amsart}

\usepackage[legalpaper, margin=1.4in]{geometry}
\usepackage[colorlinks = true,
linkcolor = blue,
urlcolor  = blue,
citecolor = blue,
anchorcolor = blue]{hyperref}  
\usepackage[english]{babel}
\usepackage{amsmath, amssymb}
\usepackage[latin1]{inputenc}
\usepackage{import}
\usepackage{epsfig}
\usepackage{color}
\usepackage{graphics}
\usepackage{relsize}
\usepackage{mathrsfs} 
\usepackage{exscale}
\usepackage{comment}
\usepackage{amsthm}

\newtheorem{theorem}{Theorem}[section]
\newtheorem{proposition}[theorem]{Proposition}

\newtheorem{claim}[theorem]{Claim}
\newtheorem{lemma}[theorem]{Lemma}
\newtheorem{corollary}[theorem]{Corollary}
\theoremstyle{definition} 
\newtheorem{ques}[theorem]{Question}
\newtheorem{remark}[theorem]{Remark}

\newenvironment{claimproof}{\paragraph{\textit{Proof of the claim}.}}{\hfill$\square$}

\theoremstyle{definition}
\newtheorem{definition}[theorem]{Definition}

\newtheorem{example}[theorem]{Example}

\numberwithin{equation}{section}
 
\def\R{\mathbb{R}}                              
\def\Z{\mathbb{Z}}                              
\def\C{\mathbb{C}}                              
\def\N{\mathbb{N}}

\def\eps{\epsilon}

\emergencystretch=\maxdimen
\begin{document}
	
	\title[Envelope of truncated tubes and special domains]
	{Envelope of truncated tubes and special domains in higher complex dimensions}

	\author{Suprokash Hazra}
\curraddr{Department of Engineering, Mathematics and Science Education (IMD), Mid Sweden University,
SE-85170 Sundsvall, Sweden} \email{suprokash.hazra@miun.se}	
	

	\begin{abstract}
	In this article, we introduce special domains and discuss the geometry of these domains, which includes showing that every pseudoconvex truncated tube domain is a special domain. Next, we prove a theorem for the envelope of special domains in $\C^n ~(n\geq 2)$. Our theorem on special domains is a generalization of a recent result by Jarnicki-Pflug on the envelope of holomorphy of truncated tube domains in $\C^n$. We also establish a result on schlichtness in complex dimension 2, and conclude this article with two  higher-dimensional generalizations of the same result by Jarnicki-Pflug mentioned above.
	\end{abstract}
	
	\maketitle
	
\section{Introduction}\label{sec_1}
For a domain $D\subset \C^n$, its envelope of holomorphy $\mathcal{E}(D)$ is known to be a {maximal} Riemann domain over $\C^n$ with projection $\pi : \mathcal{E}(D) \to \C^n$, such that there is a canonical holomorphic embedding $\iota_D:D\hookrightarrow \mathcal{E}(D)$ satisfying the lifting property $\pi\circ\iota_D=\mbox{id}_D$, and that for all $f \in \mathcal{O}(D)$, the map $\pi^* f:=f\circ(\pi|_{\iota_D(D)})$ has a holomorphic extension to $\mathcal{E}(D)$. For further details, see \cite{JP1} and the references therein. Due to a result of Oka, $\mathcal{E}(D)$ is known to be a Stein manifold. In general, the envelope $\mathcal{E}(D)$ of a domain $D$ need not be schlicht (that is, $\#\pi^{-1}(z)\neq 1$) and thus may be multisheeted over $\C^n$. Moreover, in general, it is neither always evident to determine the number of sheets of the envelope of a domain, nor always simple to comprehend how the geometry of a given domain relates to the geometry of its envelope. Nevertheless, a classical theorem due to Bochner (\cite{boch}) shows that for a domain $X\subset \R^n$, the envelope of holomorphy of the \textit{complete Bochner tube} $X+i\R^n$ is the tube $  \textsf{ch}(X) + i\R^n$ over the convex hull $\textsf{ch}(X)$, and hence the envelope is schlicht. Recall that a domain is called a \textit{truncated tube domain} (or simply a \textit{tube domain}) if it is of the form $X+iY$, where $X, Y \subset \mathbb{R}^n$ are arbitrary domains. Although the geometry of pseudoconvex tube domains is relatively well understood (see  \cite{nogu}, \cite{egm}), that of non-pseudoconvex domains remains much undeveloped. Nevertheless, recently, the following result for the non-pseudoconvex truncated tube was obtained.

\begin{theorem}[Jarnicki-Pflug, \cite{JP2}]\label{JP2theorem}
For $0\leq r_1 < r_2$ and $r_3>0$, the envelope of holomorphy of the truncated tube domain $$\{x \in \R^n: r_1<||x||<r_2\}+i\{y \in \R^n:||y||<r_3\} \subset \C^n; ~ n\geq 2,$$ is given by $$\{x+iy\in \C^n: ||x||<r_2, ||y||<r_3,  ||y||^2<||x||^2-(r_1^2-r_3^2)\}.$$ In particular, the envelope is schlicht.
\end{theorem}

Recently, S. Hazra and E. Porten proved that the envelope of truncated tube domains need not be schlicht (\cite{haz-por}). To be specific, the authors have constructed a truncated tube domain in $\C^2$ diffeomorphic to the $4$-ball such that the envelope has infinitely many sheets lying uniformly over
a neighborhood of a circle embedded in $\C^2$. In addition, they established a sufficient condition in 
$\C^2,$ which yields the following theorem.

\begin{theorem}[Hazra-Porten, \cite{haz-por}]\label{haz-por_sch_2}
If $X\subset \R^2$ be a convex domain with finitely many strictly convex holes, and $Y \subset \R^2$ be a convex domain, then the envelope of holomorphy of the truncated tube $X+ iY$ is schlicht.
\end{theorem}

In \cite[Remark 4.2.(a)]{haz-por}, the authors also mentioned that it is unknown if stronger results can be obtained for truncated Bochner tubes. In particular, this means it is uncertain if Theorem \ref{haz-por_sch_2} continues to hold in $\C^n$ for $n\geq 3$. We would like to address this as a modified question:
\begin{ques}
Let $X,Y$ be two convex domains in $\R^n$ with $n\geq 3$. Denote $K=K_x+ibY$, where $K_x\subset \R^n$ is a compact convex set. Does the envelope of holomorphy of the truncated tube $(X\setminus K_x)+iY$ always schlicht? 
\end{ques}

Observe that in $\C^2$, proving  the  pseudoconvexity in Theorem \ref{haz-por_sch_2} follows immediately from [Lemma \ref{ros-st}, Rosay-Stout], which doesn't always hold in higher dimensions. However, with an additional assumption of $b\widehat{K}$ being Levi-flat,  Theorem \ref{haz-por_sch_2} extends to $\C^n$ for $n\geq 3$ (see \cite[Theorem 3.1]{thie}). It is worth mentioning that in Theorem \ref{JP2theorem}, the condition of $b\widehat{K}$ being Levi-flat was satisfied (see \cite[Remark 4.4.(a)]{haz-por} for details). Note that, as in general, rational hull doesn't coincide with the polynomial hull, it is natural to ask about the situation of Theorem \ref{haz-por_sch_2} when $b\widehat{K}_{\mathcal{R}}$ is Levi-flat instead of $b\widehat{K}$. On the other hand since it is proved in \cite[Example 4.3]{haz-por} that convexity assumption in Theorem \ref{haz-por_sch_2} cannot be dropped, it opens up a possibiliy to ask if Theorem \ref{haz-por_sch_2} holds in $\C^2$ when $X$ has finitely many holes with convex $\mathcal{C}^2$-boundary, instead of the strictly convex holes. Finally, it is worth considering the following question.

\begin{ques}[Jarnicki/Pflug, Remark 3, \cite{JP2}]\label{ques_3}
Does every generalized tube domain with an additional geometric property (e.g. substitute the norm $||.||$ in Theorem \ref{JP2theorem} by another $\R$-norm) have a univalent envelope of holomorphy? 
\end{ques}

Surprisingly, only a very few results are known in the literature regarding the schlichtness of the envelope for truncated tube domains. Therefore, the above discussion and questions suggest that it is necessary to develop sufficient conditions that guarantee the schlichtness of the envelope for truncated tube domains in higher complex dimensions. Motivated by this aim, we now state the main focuses and outline the results of this article.\\

\noindent\textbf{(i)} We start by establishing Theorem \ref{schlicht2d}, which is a variant of Theorem \ref{haz-por_sch_2} in $\C^2$ when $X$ has finitely many holes with convex $\mathcal{C}^2$-boundary, instead of the strictly convex holes. The result is obtained by considering compacts that are CR-convex.\\

\noindent \textbf{(ii)} Next, we introduce the notion of a \textit{special domain} (Definition \ref{special_domain}) and establish some of their geometric properties. Our class of special domains includes, in particular, the pseudoconvex truncated tube domains. A special domain $S_{\alpha}$ is obtained by first considering $\alpha$ as a \textit{special barrier} with special augmenting function $f$, and later removing a \textit{compact fence} $K_{\alpha}$ given by $\alpha+f$ from the \textit{convex initial domain} $X+iY$ (Definition [\ref{special_barrier}, \ref{compact-fence}, \ref{convex-initial-domain}]). We provide a sufficient condition for the schlichtness of $\mathcal{E}(S_{\alpha})$ and prove Theorem \ref{1st_theo}, which extends Theorem \ref{JP2theorem} in $\C^n$, for $n\geq 2.$\\

\noindent \textbf{(iii)} In Theorem \ref{2nd_theo}, we provide a sufficient condition to ensure schlichtness of a truncated tube domain in higher complex dimensions. Our result extends Theorem \ref{JP2theorem} in $\C^n$ with $n\geq 2$, by applying Stout's theorem on analytic continuation. \\

\noindent\textbf{(iv)}\label{part:iv} In Section 6, we show that under additional assumptions, the schlichtness of the envelope in dimension $n\geq 3$ is guaranteed provided that $b\widehat{K}_{\mathcal{R}}$ is Levi-flat (Theorem \ref{3rd_theo}). Our theorem extends Theorem \ref{JP2theorem} in $\C^n$ with $n\geq 3$.\\

Theorem \ref{schlicht2d} may be regarded as a minor refinement, and an alternative  variant of Theorem \ref{haz-por_sch_2}. Apart from our primary focuses mentioned above, Section \ref{sec_2} is committed to preliminaries and the review of classical results, while in Section \ref{sec_3}, we revisit the Jarnicki-Pflug theorem and provide an easy and alternative proof of Theorem \ref{JP2theorem}. Finally, Section \ref{sec_4} is dedicated to the geometry of compact fences and special domains, including their examples and properties.


\section{Preliminaries and some classical results}\label{sec_2}
In this section, we briefly recall some definitions and results, which will be used in the subsequent sections. Throughout we write $x=(x_1,x_2,\ldots,x_n)\in \R^n$, $z=(z_1,z_2,\ldots,z_n)\in \C^n$, and for $q\geq 1$, $||x||_q:=(\sum_{j=1}^{n}|x_j|^q)^{1/q}$, unless specified otherwise. When $q=2$, we omit the superscript and only write $||x||$ to denote the 2-norm of $x$. For two sets $A,B$ in any metric space, denote $bA$, $\textsf{int}_AB$, and $\textsf{dist}(A,B)$ as respectively the boundary of $A$, the interior of $B$ relative to $A$, and the distance from $A$ to $B$. Recall that a point $p$ of an analytic set $X\subset \C^n$ is called \textit{regular} (denoted by $p\in \textsf{reg}(X)$) if $p$ has an open neighborhood
$U$ such that $X \cap U$ is a complex manifold. The \textit{singular set} $\textsf{sing}(X)$ is defined as $\textsf{sing}(X):=X\setminus \textsf{reg}(X).$ A critical point $p$ of a smooth real-valued function $f$ on an $n$-dimensional smooth manifold $M$ is called \textit{non-degenerate} if and only if with respect to a local coordinate system $(x_1,\ldots,x_n)$ in a neighborhood of $p$, the following matrix 
\begin{equation}\label{hessian}
	\Big(\frac{\partial^2f}{\partial {x_j}\partial x_k}(p)\Big)
\end{equation}

\noindent is non-singular. A function $f : M \to \R$ is called a \textit{Morse function} if its critical points are non-degenerate.
For more details on Morse theory, see \cite{milnor}.

\begin{definition}[page 120, \cite{kra}]\label{def-st-convexity}
For a domain $G \subset \R^n$ with smooth boundary, and a locally defining function $r$ near a neighborhood $U$ of a point $p \in bG$ with $G \cap U\subset\{x\in \R^n : r(x)<0\}$, the notion of \textit{strict convexity} at $p \in bG$ is equivalent to the following condition $$\sum \limits_{j,k =1}^{n} {a_j} a_k  \frac{\partial^2r}{\partial {x_j}\partial x_k}(p) > 0, \text{ for all } X_p = \sum \limits_{j =1}^{n} a_j \frac{\partial}{\partial x_j}\Bigr\rvert_p \in T_{p}(bG).$$
A domain $G$ is strictly convex if it is strictly convex at all $p \in bG$.	
\end{definition}
 It is elementary to verify that every bounded strictly convex domain is convex, but not conversely (e.g. $\{(x_1,x_2)\in \R^2: x_1<1\}$). \\

\begin{definition}[CR manifold and CR function, \cite{bogg}]
	Let ${J :T_p \C^n \to T_p \C^n}$ be a linear map with $J^2=-I$, and let $M$ be a smooth submanifold of $\C^n$. If the real dimension of  $T_p^JM :=T_pM \cap J(T_pM)$ is independent of $p\in M$, we call $M$ a \textit{CR manifold}. A function $f:M\to \C$ is called a \textit{CR function} if for all $p\in M$ and for all $\overline{L_p}\in T_p^{(0,1)}M$ it implies $\overline{L_p}f=0$, where
	$$T_p^{(0,1)}\C^n:= \textsf{span}_{\C}\big\{\frac{\partial}{\partial \overline{z_j}}:1\leq j\leq n\big\};~ \text{ and }~ T_p^{(0,1)}M:= (\C \otimes T_pM)\cap T_p^{(0,1)}\C^n. $$
\end{definition}

\begin{definition}[Pseudoconvexity and Levi-flatness, \cite{kra}]
Let $\Omega\subset \C^n$ be an open set with $\mathcal{C}^2$-boundary such that $\rho$ is a locally defining function for $b\Omega$ at $p\in b\Omega$ with $\rho|_{\Omega}(\cdot)<0$. For an element $v_p=\sum_{j =1}^{n} w_j \frac{\partial}{\partial z_j}\Big\rvert_p \in T_{p}^{(1,0)}(b\Omega)$ the following sesquilinear form $$\mathcal{L}(v_p, v_p):=\sum \limits_{j,k =1}^{n} {w_j}\overline{w_k} \frac{\partial^2\rho}{\partial {z_j}\partial \overline{z}_k} (p)$$ is called the \textit{Levi form} at $p.$ The set $ \Omega$ is called  \textit{pseudoconvex} (resp. \textit{Levi flat}) at $p\in b\Omega$ if the Levi form is positive semidefinite (resp. zero). A smooth real hypersurface $M\subset \C^n$ is called Levi-flat if its Levi form vanishes at every point of $M$ (see \cite{bogg} for details). 
\end{definition}

\begin{theorem}[\cite{bogg}, chapter 10, Theorem 1]\label{foli}
	Suppose $(M,T^JM)$ be a Levi flat CR structure. Then $M$ is locally foliated by complex manifolds whose complexified tangent bundle is given by $(T^JM\oplus \overline{T^JM}).$
\end{theorem}

Recall that a complex manifold is called \textit{Stein} if it is holomorphically separable (that is, any two distinct points are separated by a holomorphic function) and holomorphically convex (that is, the holomorphic convex hull of a compact set is compact). Next we state the following maximum-modulus principle.

\begin{theorem}[Theorem 16, Chap. III, \cite{gun-ros}]\label{gun-ross}
Let $V$ be a connected subvariety of the domain $G\subset \C^n$ with $n\geq 2$. If the modulus of any $f\in \mathcal{O}(G)$ attains a maximum in $V$, then $f$ is con­stant on $V$.	
\end{theorem}

\subsection{Different hulls and some classical results}
For a bounded domain ${\Omega \subset \C^n}$, let $\mathcal{O}(\Omega)$, $\mathsf{psh}(\Omega)$ and $\mathcal{A}(\Omega)$ respectively denote the set of all holomorphic functions defined on $\Omega$, the set of plurisubharmonic functions on $\Omega$, and the set of all $f\in \mathcal{O}(\Omega)$ that are continuous up to $\overline{\Omega}$. 
Next we use the following notations to describe the hulls (see \cite{jor2}, \cite{stout} for details).
\vspace{.2cm}
\begin{itemize} 
\item [(i)] ${\widehat{K}}_{\mathcal{A}(\Omega)}:=\{z\in\C^n:|f(z)|\leq\max_{K}|f|,  \mbox{ for all } f\in \mathcal{A}(\Omega) \}.$
\item [(ii)] $\widehat{K}_{\mathsf{psh}(\Omega)}:=\{z\in \C^n :\phi(z)\leq \sup_K \phi, \mbox{ for all } {\phi \in \mathsf{psh}(\Omega)}\}.$
\item [(iii)]
$\widehat{K}:=\{z\in\C^n:|p(z)|\leq\max_{K}|p|,\text{ for all } p \in \C[z_1,z_2,\dots, z_n]\}.$
\item [(iv)]  $\widehat{K}_{\mathcal{R}}:=\{z \in \C^n: p^{-1}(0)\cap K\neq \emptyset, \forall p \in \C[z_1,z_2,\dots, z_n] \text{ with } p(z) = 0\}.$
\end{itemize}
\vspace{.2cm}
The hulls mentioned in (i), (ii), (iii), and (iv) are respectively called the {$\mathcal{A}(\Omega)$-hull},  $\mathsf{psh}$-hull, polynomially convex hull, and the rationally convex hull. A compact set $K$ is said to be respectively \textit{polynomially convex} or \textit{rationally convex} if $\widehat{K}=K$, and $\widehat{K}_{\mathcal{R}}=K$. \\

The next lemma is a special case of  \cite[Theorem 2.1, p. 365]{slowd}.

\begin{lemma}[Rosay-Stout, \cite{ros-st}]\label{ros-st}
Let $K \subset \C^2$ be a compact set, and let $\Omega \subset \C^2$ be a pseudoconvex domain. If $\Omega \cap K=\emptyset$, then the set $\Omega \setminus \widehat{K}$ is pseudoconvex.
\end{lemma}

\begin{definition}[\cite{stout}, \cite{jor2}]
For a pseudoconvex domain $\Omega$, a compact set $K \subset b\Omega$ with $K \neq b\Omega$ is said to be  \textit{{$\mathcal{A}(\Omega)$-convex}} if ${\widehat{K}}_{\mathcal{A}(\Omega)}=K$. Moreover, $K$ is said to be \textit{$CR(b\Omega)$-convex} if $K = b\Omega \cap {\widehat{K}}_{\mathcal{A}(\Omega)}$, although $K$ is not necessarily $\mathcal{A}(\Omega)$-convex.
\end{definition}
\begin{theorem}[J\"oricke, \cite{jor2}]\label{jor2}
	Let $\Omega$ be a bounded pseudoconvex domain in $\C^2$ with boundary $b\Omega$ of class
	$\mathcal{C}^2$. Suppose $K \subset b\Omega$ with $K \neq b\Omega$ is a compact $CR(b\Omega)$-convex set. Let for a one-sided neighborhood $V$ of $b\Omega \setminus K$ with $V \subset \Omega \setminus {\widehat{K}_{\mathcal{A}(\Omega)}}$, each
	connected component of $V$ contains in its boundary exactly one component of $b\Omega \setminus K$
	and no other point of $b \Omega \setminus K$. Then any $f\in \mathcal{O}(V)$ has (uniquely
	determined) analytic extension to $\Omega \setminus {\widehat{K}}_{\mathcal{A}(\Omega)}$.		
\end{theorem}

\begin{definition}[$\mathcal{R}(\overline{D})$-convex and $\mathcal{O}(\overline{D})$-convex, \cite{st-81}]\label{ODRD-cvx}
Let $D\subset \C^n$ be a bounded domain of holomorphy. A closed set ${E \subset bD}$ is called  $\mathcal{R}(\overline{D})$-convex if for $p\in \overline{D}\setminus E$, there is $h\in \mathcal{O}(\overline{D})$ with $h(p)=0$ and $h|_E \neq 0$. $E$ is said to be  $\mathcal{O}(\overline{D})$-convex if given $p\in \overline{D}\setminus E$ there is $f\in \mathcal{O}(\overline{D})$ with $f(p)=1$, and $||f||_E=\sup_E|f|<1$. 
\end{definition}
Note that if $E$ is rationally convex, then it is $\mathcal{R}(\overline{D})$-convex. Moreover, if $\overline{D}$ in Definition \ref{ODRD-cvx} is polynomially convex, then the condition of $\mathcal{O}(\overline{D})$-convexity is clearly the condition that $E$ is polynomially convex (see \cite{st-81}).
\begin{theorem}[Stout, \cite{st-81}]\label{stout_theo}
	Let $D$ be a bounded domain of holomorphy in $\C^n$ with $n\geq 2$, let $E\subset bD$ be closed, and let $\Omega \subset D$ be an open set such that $\Omega \cup ( bD\setminus E)$ is a
	neighborhood of $ bD\setminus E$ in $\overline{D}$. Then if either of the conditions is satisfied, 
	\begin{itemize}
		\item [(i)] $n\geq 2$ and $E$ is $\mathcal{O}(\overline{D})$-convex,
		\item [(ii)] $n\geq 3$ and $E$ is $\mathcal{R}(\overline{D})$-convex,
	\end{itemize}
	
	then, for any $f\in \mathcal{O}({\Omega})$, there exists an $F\in \mathcal{O}({D})$ with $F|_{\Omega} = f$ near $bD$. Moreover, if $\Omega$ is connected, then $\mathcal{E}(\Omega)=D.$
\end{theorem}
We end this section by recalling the following definition.
\begin{definition}[Page 164, \cite{stout}]\label{var_cvx}
	A compact subset $E$ of $\C^n$ is said to be convex with respect to varieties of dimension $m$ if for each point $p\in \C^n \setminus E$ there is a purely $m$-dimensional subvariety 
	$V_p$ of $\C^n$ containing $p$ such that $E\cap V_p=\emptyset$.
\end{definition}


\section{Revisiting  two classical results on envelope of holomorphy}\label{sec_3}

We start this section by recalling that the proof of Theorem \ref{JP2theorem} involves a result by Iva\v skovi\v c in complex dimension 2 (\cite[Lemma 8]{ivas}), followed by applying a slightly constructive technique in complex dimension $n\geq 3$. Here we first establish an  alternative and concise proof of Theorem \ref{JP2theorem} using Stout's theorem on analytic continuation.\\

\textbf{Alternative proof of Theorem \ref{JP2theorem}.}
We start with the following pseudoconvex bounded domain $$D=\Big\{x+iy\in \C^n:||x||^2-||y||^2>(r_1^2-r_3^2)\Big\}\cap \Big(B_{r_2}(0)+iB_{r_3}(0)\Big).$$
Set $E:=\{x+iy\in \overline{D}:||x||^2-||y||^2=(r_1^2-r_3^2)\} \subset \partial D.$ To prove $E$ is $\mathcal{O}(\overline{D})$-convex,  
we choose a $p=x_0+iy_0\in \overline{D}\setminus E$. 
Note that $p$ satisfies ${||x_0||^2-||y_0||^2>(r_1^2-r_3^2)}$. Choose $\epsilon_0>0$ such that $$||x_0||^2-||y_0||^2=(r_1^2-r_3^2)+\epsilon_0.$$
Define the entire function $h(z):=\exp({\sum_{j=1}^{n}z_j^2})/\exp({r_1^2-r_3^2+\epsilon_0})$ for all $z\in \C^n$. Notice that $|h(p)|=1$ and ${||h||_E=\exp(r_1^2-r_3^2)/\exp({r_1^2-r_3^2+\epsilon_0})= \frac{1}{\exp(\epsilon_0)}<1}.$
Since $|h(p)|=1$, there exists $\theta\in \R$ with $h(p)=\exp(i\theta)$. Next, consider the entire function $f:\C^n \to \C$ defined by ${f(z):=\exp(-i\theta)h(z)}$ for all $z\in \C^n$, and note that ${f(p)=\exp(-i\theta)\exp(i\theta)=1.}$ Moreover, as the modulus remains unchanged under rotation, it implies $||f||_E=\sup_E|f|<1$ too. Thus $E$ is $\mathcal{O}(\overline{D})$-convex by Definition \ref{ODRD-cvx}. Finally, we consider the connected domain $$\Omega=\{x \in \R^n: r_1<||x||<r_2\}+i\{y \in \R^n: ||y||<r_3\}$$ and apply Theorem \ref{stout_theo} to conclude that $\mathcal{E}(\Omega)=D.$\\

By a \emph{domain $X$ convex up to exclusion of finitely many convex holes} we mean a domain $X\subset\R^n$ obtained from a convex domain $X_{\text{cvx}}$ by excluding finitely many disjoint compact convex sets $\overline{H}_{x,j}$, $1\leq j\leq m_0$. Next, we prove the following lemma which is a minor variant of \cite[Lemma 4.1]{haz-por}. The key observation here is that one can relax the assumption of \textit{strict convexity} in \cite[Lemma 4.1]{haz-por} by adding more regularity in the $y$-direction (e.g., assuming $bY$ to be $\mathcal{C}^2$-smooth). In view of this, we obtain the following lemma.

\begin{lemma}\label{cvx-lem}
	Let $X\subset\R^2$ be a domain convex up to exclusion of a single convex hole $H_x$ with $\mathcal{C}^2$-boundary (and not necessarily strictly convex), and let $Y\Subset\R^2$ be a convex domain with $\mathcal{C}^2$-boundary. Denote  ${H=H_{x}+ibY}$. Then $\mathcal{E}(D)=D'\backslash\widehat{H}$, where $D=X+iY$ and $D'=(X\cup H_x)+iY$.
\end{lemma} 

\begin{proof}
$D'$ is a bounded pseudoconvex domain in $\C^2$ but not of class $\mathcal{C}^2$. Therefore we perform a very small $\mathcal{C}^2$-deformation to the ``outer boundary" (i.e., to the points of $bD'\setminus H$) of the domain $D'$ to obtain a new pseudoconvex domain $D_{\text{new}}$ of class $\mathcal{C}^2$ such that $D_{\text{new}}\subset D'$. Since $(H_x+i\overline{Y}) \cap bD'=H,$ it implies that $D_{\text{new}}$ can be chosen with ${(H_x+i\overline{Y}) \cap bD_{\text{new}}=H}.$ By convexity of $H_x+i\overline{Y}$, it implies ${\widehat{H}}_{\mathcal{A}(D_{\text{new}})}\subset H_x+i\overline{Y}$, and hence  ${{\widehat{H}}_{\mathcal{A}(D_{\text{new}})}\cap bD_{\text{new}}=H.}$  Therefore $H$ is $CR(bD_{\text{new}})$-convex. Note that the following one-sided neighborhood $$V= \big(bD_{\text{new}} \setminus H\big) \cup \big(D_{\text{new}}\cap D\big) \subset \overline{D}$$ of $bD_{\text{new}} \setminus H$ satisfies the conditions of the Theorem \ref{jor2}. Therefore it follows that every holomorphic function defined on $D$ (and on $D_{\text{new}}\cap D$ by restriction) has a (uniquely
determined) analytic extension to $D_{\text{new}} \setminus {\widehat{H}}_{\mathcal{A}(D_{\text{new}})}$ and hence to ${D' \setminus {\widehat{H}}_{\mathcal{A}(D_{\text{new}})}}$ by construction. Next we see that
\begin{equation}\label{ineq_diff_hulls_same}
D'\setminus \widehat{H}\subseteq D'\setminus{\widehat{H}}_{\mathcal{A}(D')}\subseteq D'\setminus {\widehat{H}}_{\mathcal{A}(D_{\text{new}})},
\end{equation}
as $\mathcal{O}(\C^n)\subseteq \mathcal{A}(D')\subseteq \mathcal{A}(D_{\text{new}}).$ Since $H$ satisfies $D' \cap H=\emptyset$, by {Lemma \ref{ros-st}} it follows that $D'\setminus \widehat{H}$ is pseudoconvex. Finally, due to pseudoconvexity, all the sets in equation (\ref{ineq_diff_hulls_same}) are equal, which completes the proof.
\end{proof}

Next, we state the following variant of Theorem \ref{haz-por_sch_2} omitting repetitive details of the proof, as it closely follows the original argument. 

\begin{theorem}\label{schlicht2d}
Let $X\subset\R^2$ be a domain convex up to exclusion of finitely many convex holes with $\mathcal{C}^2$-boundary (and not necessarily strictly convex) and let ${Y\Subset\R^2}$ be a convex domain with $\mathcal{C}^2$-boundary. Then the envelope of holomorphy of $D=X+iY$ is schlicht.
\end{theorem}
\begin{proof}
The proof of the general case with finitely many convex holes, follows similary as mentioned in \cite[Page 2991]{haz-por} by smoothly deforming the holes and applying Behnke-Stein theorem (\cite{beh-ste}) to the increasing union of pseudoconvex sets.
\end{proof}

We now shift our attention to different hulls. Recall that a polynomially convex compact set is rationally convex, but not conversely. In general, it is still an open problem to know when a rationally convex compact set $K$ with 1-dimensional \v Cech cohomology group $\check{H}^1(K,\Z)=0$ is polynomially convex (see \cite[Question II.7]{stol2}). Although explicit calculations of the hulls are hard in general, nevertheless, in certain contexts $\widehat{K}_{\mathcal{R}}$ and $\widehat{K}$ may become the same, as follows from our next proposition.  

\begin{proposition}\label{both_hull_same}
	Let $r_1,r_3\in \R^+$, $n\geq 2$, and $K_x=\{x\in \R^n: ||x||\leq r_1\}$. Denote $Y=\{y\in \R^n: ||y||<r_3\}$  and $K=K_x+ibY$.  Then $\widehat{K}_{\mathcal{R}}=\widehat{K}$, and it is given by $$\big\{x+iy\in \C^n: ||y||\leq r_3, ||x||^2-||y||^2\leq (r_1^2-r_3^2)\big\}.$$
\end{proposition}

\begin{proof}
	In $\C^n$ with $ n\geq 2$, it can easily be derived from \cite[Example 4.4(a)]{haz-por} that $$\widehat{K}=\{x+iy\in \C^n: ||y||\leq r_3, ||x||^2-||y||^2\leq (r_1^2-r_3^2)\}.$$
	We first want to show $$N:=\{x+iy\in \C^n: ||y||\leq r_3, ||x||^2-||y||^2< (r_1^2-r_3^2)\} \subset \widehat{K}_{\mathcal{R}}.$$ Suppose not so, and let there be a point $z_0=x_0+iy_0\in N\setminus \widehat{K}_{\mathcal{R}}$. By definition there is a holomorphic polynomial $P:\C^n\to \C$ such that $z_0\in P^{-1}(0)$ and $P^{-1}(0) \cap K=\emptyset.$ Due to $P^{-1}(0)$ being an $(n-1)$-dimensional analytic variety in $\C^n$, by \cite[Corollary 17, Chap. III]{gun-ros}, it therefore cannot be compact. Thus, $P^{-1}(0)$ is unbounded. Since $K\cap P^{-1}(0)=\emptyset$, the set $P^{-1}(0)$ must therefore intersect ${E:=\{||x||^2-||y||^2= (r_1^2-r_3^2)\}}$. 
	It implies ${||x_0||^2-||y_0||^2< (r_1^2-r_3^2)}$, due to that $z_0\in N$.	Choose $\eps>0$ satisfying $||x_0||^2-||y_0||^2= (r_1^2-r_3^2)-\eps,$ and consider the holomorphic function $\varphi(z):= \exp{(r_1^2-r_3^2-\eps)}/\exp{(\sum_{j=1}^{n}z_j^2)}.$ Next, we take an irreducible branch $V'$ of $P^{-1}(0)\cap N$ that contains $z_0$. 
	Clearly, it implies $|\varphi(z_0)|=1$ and  $$\max\limits_{bV'}|\varphi|= \frac{\exp{(r_1^2-r_3^2-\eps)}}{\exp{(r_1^2-r_3^2)}}=\frac{1}{\exp(\eps)}<1.$$ 
	However, since $V'$ is a connected subvariety of the domain $N$ containing $z_0$, by {Theorem \ref{gun-ross}} the modulus $|\varphi|$ cannot attain maximum on $V'.$ This leads to a contradiction. Consequently $N\subset \widehat{K}_{\mathcal{R}}$, and the result follows by taking closures on both sides.
\end{proof}

We end this section by noting that the hull mentioned in the above proposition was crucial, and it was used in \cite[Example 4.4 (a)]{haz-por} to show a generalization of Jarnicki-Pflug's result (\cite{JP2}) in $\C^2$. Moreover, we prepared this proposition for using it later in {Section \ref{sec_6}} to obtain Theorem \ref{JP2theorem} as a corollary of our {Theorem \ref{3rd_theo}}.


\section{Geometry of compact fence and special domain}\label{sec_4}
In this section, we introduce and study the notion of \textit{compact fence} and a new type of domain, namely the \textit{special domain}. We start by letting $\alpha: \R^n\to \R$ be a \textit{globally} defined $\mathcal{C}^{\omega}$ function. Clearly, there is a holomorphic function $F_{\alpha}$ in complex variables $z_j,~ 1\leq j\leq n$, (defined at least on some open neighborhood of $\R^n$ in $\C^n$) and a real-valued power series $f_{\alpha}(x,y)$ such that $F_{\alpha}|_{\R^n}(x+iy)=\alpha(x),$ and  $\textsf{Re}F_{\alpha}(x+iy)=\alpha(x)+f_{\alpha}(x,y)$ hold for all $x,y \in \R^n$. For notational simplicity, $f$ is written throughout instead of $f_{\alpha}$ when there is no ambiguity in the context. We call $f$ the \textit{augmenting function} of $\alpha$. Since it is known that $F_{\alpha}$ mentioned above need not always be holomorphic on $\C^n$, this situation needs to be different from our next definition. 

\begin{definition}[good $\mathcal{C}^{\omega}$ and good augmenting function]\label{good_alpha}
A $\mathcal{C}^{\omega}$ function $\alpha:\R^n \to \R$ is called a \textit{good $\mathcal{C}^{\omega}$ function} if there is a holomorphic function ${F_{\alpha}:\C^n \to \C}$ such that $\textsf{Re}F_{\alpha}(x+iy)=\alpha(x)+f(x,y)$ for all $x,y \in \R^n$. The corresponding augmenting function $f$ is called a \textit{good augmenting} function.
\end{definition}
Examples of good $\mathcal{C}^{\omega}$ functions are in particular, the real polynomials in the variables $x_j, ~1\leq j\leq n$. In the literature, it is conventional to first define a holomorphic function $F$ and later obtain $\textsf{Re}(F)$ from it, unlike our situation. The reason to first start with $\alpha$ 
is that we want to distinguish $\alpha $ and its corresponding augmenting function $f$ to define the notion of \textit{special barrier}. This reason would become clearer when we introduce our next definition.
\begin{definition}[Special barrier]\label{special_barrier}
A good $\mathcal{C}^{\omega}$ function ${\alpha: \R^n\to \R}$ (with the good augmenting function $f$) is called a \textit{special barrier} with respect to a convex set $Y\subset \R^n$, if for all fixed $y_0\in Y$ the following real Hessian matrix
\begin{equation}\label{matrix_cond}
\begin{pmatrix}
\frac{\partial^2\Big(||x||^2-\big(\alpha(x)+f(x,y_0)\big)\Big)}{\partial x_j\partial x_k}(\textbf{x})
\end{pmatrix}_{1\leq j,k \leq n}\end{equation}
is positive definite for all $\textbf{x}\in \R^n.$ In this case, we call $f$ a \textit{special augmenting} function.

\end{definition}
Before proceeding, we collect some simple examples of special barrier $\alpha$.

\begin{example}\label{ex_special_bar}
The following functions are special barriers with respect to any convex set $Y\subset \R^n.$
\begin{itemize}
\item [(i)] Let $c$ and $r_1$ be any two real constants. Define the two functions  ${\alpha_{\text{const}}:\R^n \to \R}$ and $\alpha_{\geq0}:\R^n \to \R$ respectively by $\alpha_{\text{const}}(x):=c$, and $\alpha_{\geq0}(x):=r_1^2$ for all $x\in \R^n$. Note that in both cases the special augmenting functions are equal to $0$. Therefore both $\alpha_{\text{const}}$ and $\alpha_{\geq0}$ are special barriers with respect to any convex set $Y\subset \R^n.$
\item [(ii)] For $\delta_j\in \R$ the function $\alpha_{\text{lin}}(x):=\sum_{j=1}^{n}\delta_jx_j$ is a special barrier with respect to any convex set $Y\subset \R^n.$ Note that in this case, the special augmenting function is zero too.

\item[(iii)] Consider the good $\mathcal{C}^{\omega}$ function $\alpha_{\text{mix}}(x):=\sum_{j=1}^{n-1}x_jx_{j+1}$, and note that the corresponding good augmenting function $f(x,y)=-\sum_{j=1}^{n-1}y_jy_{j+1}$. The matrix mentioned in equation (\ref{matrix_cond}) in this case, is the tridiagonal matrix with $-1, 2$, and $-1$ as entries. Due to a result by J. F. Elliott (see \cite{elli}), it is known that the eigen values of this matrix are all positive and given by $$4 \sin^2\Big(\frac{k\pi}{2(n+1)}\Big); ~ k=1,2,\cdots,n.$$
 Therefore the function $\alpha_{\text{mix}}(x)$ is a special barrier with respect to any convex set $Y\subset \R^n.$
 
\item[(iv)] If $\alpha_1, \alpha_2$ are special barriers  with respect to the convex set $Y_1$ and $Y_2$ respectively, then for any $t\in [0,1]$, the function $\alpha:=t\alpha_1+(1-t)\alpha_2$ is also a special barrier with respect to the convex set $Y_1\cap Y_2$. Note that if $f_1,f_2$ are the special augmenting functions of $\alpha_1,\alpha_2$ respectively, then the same of $\alpha$ is $tf_1+(1-t)f_2$. 
\end{itemize}	
\end{example}

Some more examples of special barriers are given in the following proposition.
\begin{proposition}\label{prop_special_bar}
Let $Y$ be any convex set in $\R^n.$ Then any good $\mathcal{C}^{\omega}$ function $\alpha$ with the corresponding good augmenting function $f$, satisfying the following two properties
\begin{itemize}
	\item[(a)] $\frac{\partial^2\big(\alpha(x)+f(x,y_0)\big)}{\partial x_j^2}(\mathbf{x})<2$ for all $\mathbf{x}\in \R^n$, $y_0\in Y$ and $1\leq j\leq n$,
	\item [(b)] $\frac{\partial^2\big(\alpha(x)+f(x,y_0)\big)}{\partial x_j\partial x_k }(\mathbf{x})=0$ for all $\mathbf{x}\in \R^n$, $y_0\in Y$ and $ j\neq k$, $1\leq j,k\leq n$,
\end{itemize}
is a special barrier with respect to $Y$. \\

In particular, if $N_0,\beta_j, \delta_j\in \R$ with ${1\leq j\leq n}$ such that $N_0\beta_j<1$ for all $1\leq j\leq n$, then the function $\alpha:\R^n \to \R$ defined by
$$ \alpha(x):=N_0\big(\sum_{j=1}^{n}\beta_jx_j^2-\sum_{j=1}^{n}\delta_jx_j\big) $$ is a special barrier with respect to the open ball $B_R(0)$, for all $R\in \R^+.$
\end{proposition}

\begin{proof}
	The first part follows from the observation that the following matrix
\begin{equation*}
A=
	\begin{pmatrix}
	2-\frac{\partial^2(\alpha(x)+f(x,y_0))}{\partial x_1^2}(\mathbf{x}) & 0  &\ldots& 0\\
	0 & 2-\frac{\partial^2(\alpha(x)+f(x,y_0))}{\partial x_2^2}(\mathbf{x}) &\ldots& 0\\	
\vdots & \vdots &\vdots& \vdots\\
0 & 0 &\ldots& 2-\frac{\partial^2(\alpha(x)+f(x,y_0))}{\partial x_n^2}(\mathbf{x})	
\end{pmatrix}
\end{equation*}
has all positive eigenvalues given by $\lambda_j:=\big(2-\frac{\partial^2(\alpha(x)+f(x,y_0))}{\partial x_j^2}(\mathbf{x})\big)$ for all $\mathbf{x}\in \R^n$, $y_0\in Y$, and $1\leq j\leq n.$ This implies that for all $\textbf{v}:=(v_1,\ldots,v_n)\in \R^n\setminus\{0\}$, the quantity $\textbf{v}A \textbf{v}^t=\sum_{j=1}^n \lambda_j v_j^2>0.$ Hence, the matrix mentioned in equation (\ref{matrix_cond}) is positive definite. The particular case follows by first letting $Y=B_R(0)$ and noting that
	$$F_{\alpha} (z)=N_0(\sum_{j=1}^{n}\beta_jz_j^2-\sum_{j=1}^{n}\delta_jz_j),$$ and $$\textsf{Re}(F_{\alpha})(x,y)=N_0\Big(\sum_{j=1}^{n}(\beta_jx_j^2-\delta_jx_j)-\sum_{j=1}^{n}\beta_jy_j^2\Big).$$
	Thus the augmenting function $f(x,y)=-N_0\sum\limits_{j=1}^{n}\beta_jy_j^2$, and consequently 
	\begin{itemize}
		\item[(a)] $\frac{\partial^2\big(\alpha(x)+f(x,y_0)\big)}{\partial x_j^2}(\mathbf{x})=2N_0\beta_j<2,~\text{for all } \mathbf{x}\in \R^n, ~y_0\in B_R(0).$\\
		\item [(b)] $\frac{\partial^2\big(\alpha(x)+f(x,y_0)\big)}{\partial x_j\partial x_k }(\mathbf{x})=0, ~ \text{for all } \mathbf{x}\in \R^n,~  y_0\in B_R(0).$\\
		
	\end{itemize}
	Thus $\alpha$ defined above is a special barrier with respect to $B_R(0)$ for all $R\in \R^+.$ 
\end{proof}

In order to define special domain, we introduce the following two notions.
\begin{definition}[Compact fence]\label{compact-fence}
For a special barrier $\alpha$ with respect to a convex domain $Y\subset \R^n, n\geq 2$, and the corresponding special augmenting function $f$, we denote 
\begin{equation}\label{kalpha}
	K_{\alpha} := \{x+iy \in \R^n+i \overline{Y} : ||x||^2\leq \alpha(x)+f(x,y)\}\subset \C^n.
\end{equation}
We call $K_{\alpha}$ the \textit{compact fence given by} $\alpha+f$. 
\end{definition}
\begin{definition}\label{convex-initial-domain}
Let $\alpha, K_{\alpha}$, and $Y$ as in Definition \ref{compact-fence}. We call a bounded \textit{strictly convex} domain $X$ in $\R^n$ with $\mathcal{C}^2$ boundary a \textit{convex initial domain} if 
\begin{equation}\label{cvx_initial}
 K_{\alpha} \cap \big(bX+i \overline{Y}\big) =\emptyset.
\end{equation}
\end{definition}
Note that if $X$ is any convex initial domain, then any bounded strictly convex domain $X'$ with $\mathcal{C}^2$ boundary containing $X$ is also a convex initial domain.

\begin{definition}[Special domain]\label{special_domain}
For a special barrier $\alpha$ with respect to a convex domain $Y$, and the special augmenting function $f$, let $K_{\alpha}$ be the compact fence given by $\alpha+f$.  Then for a convex initial domain $X$ we call the following domain
	\begin{equation}\label{salpha}
		S_{\alpha}:=(X+iY)\setminus K_{\alpha}
	\end{equation}
a \textit{special domain} in $\C^n.$ 
\end{definition}
The following proposition motivates us to study special domains further.
\begin{proposition}
Every pseudoconvex truncated tube domain is a special domain for some special barrier $\alpha$. The converse is false.
\end{proposition}
\begin{proof}
By \cite[Theorem 1.1]{egm}, every pseudoconvex truncated tube domain $D$ is of the form $X+iY$, where both $X$ and $Y$ are convex domains in $\R^n.$ Choose $c\in \R\setminus \{0\}$, and let $K_{\alpha_{-}}$ be the compact fence obtained from the special barrier $\alpha_{-}$ defined by $\alpha_{-}(x):=-c^2$ for all $x\in \R^n$. Note that $K_{\alpha_{-}}$ in this case is an empty set, and thus $X$ is a convex initial domain by Definition \ref{convex-initial-domain}. This shows $$D=(X+iY)\setminus \emptyset=(X+iY)\setminus K_{\alpha_{-}},$$ 
and therefore a special domain for the special barrier $\alpha_{-}.$

For the converse, we first prove that every special domain need not be a truncated tube domain. Let $X=B_{r_2}(0)$, $Y=B_{r_3}(0)$ be the two open balls in $ \R^n $ of radii $r_2,r_3\in \R^+$ and centered at $0.$ Next, we consider the following special barrier $\alpha(x):=\frac{1}{2}\sum_{j=1}^{n}x_j^2$ (Proposition \ref{prop_special_bar}) with $\beta_j=1,~\delta_j=0$ and $N_0=\frac{1}{2}.$ From the proof of the same proposition, it follows that the augmenting function $f(x,y)=-\frac{1}{2}\sum_{j=1}^{n}y_j^2.$ Thus, the corresponding compact fence
\begin{align*}
K_{\alpha} &:=\{x+iy \in \R^n+i \overline{Y} : ||x||^2\leq \frac{1}{2}\sum_{j=1}^{n}x_j^2-\frac{1}{2}\sum_{j=1}^{n}y_j^2\}=\{0+i0\}.
\end{align*}
It implies that the special domain $S_{\alpha}:=(B_{r_2}(0)+iB_{r_3}(0))\setminus \{0\}$ is not a truncated tube domain. To verify this, we suppose $S_{\alpha}=U+iV$ for some domains ${U,V\subset \R^n}.$ For $x,y\neq0$, both the points $z=0+i\frac{r_3}{2||y||}y$ and $z'=\frac{r_2}{2||x||}x+i0$ belong to $S_{\alpha}=U+iV$, and hence $0$ must be in $U\cap V$. However, due to $0+i0\notin S_{\alpha}$, it leads to a contradiction. 

Next, we prove that every special domain need not be pseudoconvex. To verify this, we let $X,Y$ be as above, and let the special barrier be $\alpha_{\geq0}(x):=r_1^2~(r_1\neq 0)$ for all $x\in \R^n$ (see Example \ref{ex_special_bar}.(i)). Since the augmenting function in this case is zero, it implies $K_{\alpha_{\geq0}}=\{x\in \R^n: ||x||\leq r_1^2\}+i\overline{B_{r_3}(0)}$. Therefore $$S_{\alpha_{\geq0}}=\{x \in \R^n: r_1<||x||<r_2\}+i\{y \in \R^n:||y||<r_3\}.$$
Since by Theorem \ref{JP2theorem} the envelope $\mathcal{E}(S_{\alpha_{\geq 0}})$ is strictly larger than $S_{\alpha_{\geq0}}$, the result follows. 
\end{proof}
\begin{remark}
Observe that, a special domain $S_{\alpha}$ is a truncated tube domain if the special augmenting function $f$ (corresponding to special barrier $\alpha$) is independent of the variable $y$, that is if $f(x,y)=\phi(x)$ for all $x,y\in \R^n$, and for some function $\phi:\R^n\to \R$.
\end{remark}

In \cite{haz-por}, the notion of strict convexity (see in Definition \ref{def-st-convexity}) in the $x$-direction was crucial to prove their theorem of schlichtness in complex dimension 2. Here at this point, we emphasize that our condition of $\alpha$ being a special barrier is also very natural and valid. The legitimacy of it becomes clearer in the next remark.

\begin{remark}\label{remark_strcvx}
For a fixed $y'\in \overline{Y} $, denote $$K_{\alpha}^{y'}:=\{x+iy' \in \R^n+i \overline{Y} : ||x||^2< \alpha(x)+f(x,y')\}.$$ If $bK_{\alpha}^{y'}$ is $\mathcal{C}^{\infty}$-smooth, then for any tangent vector ${\sum_{j=1}^{n} a_j \frac{\partial}{\partial x_j}\Big|_{\xi} \in T_{\xi}(b K_{\alpha}^{y'})\setminus \{0\}}$ with $\xi\in bK_{\alpha}^{y'}$ 
and real functions $a_j:\R^n\to \R$, the quantity $\textbf{a} H\textbf{a}^t>0$, where $\textbf{a}:=(a_1,\ldots, a_n)$ and $H$ is the matrix mentioned in equation (\ref{matrix_cond}). 
Therefore by Definition \ref{def-st-convexity}, the set $K_{\alpha}^{y'}$ is strictly convex. 
\end{remark}

The next proposition shows that under certain conditions on the special augmenting function $f$, the compact fence $K_{\alpha}$ (corresponding to a special barrier $\alpha$) is of a nice nature. To be precise, if $f$ has no ``mixed $xy$-terms" (defined later), and if $\frac{\partial^2 f}{\partial y_j^2}(x,y)\neq 0$ for all $(x,y)\in \R^n\times \R^n$, then $K_{\alpha}$ can be written as a union of sets $\mathcal{K}_{s}^{\alpha}$, each of which is a real hypersurface of class $\mathcal{C}^{\infty}$ in $\C^n$ except possibly, at the set of isolated singularities $\textsf{sing}(\mathcal{K}_{s}^{\alpha})$. It is worth noting that the conditions on $f$ mentioned in Proposition \ref{lem_crit} are satisfied for all the special functions $\alpha$ mentioned in Examples \ref{prop_special_bar} when $N_0\neq 0$ and $\beta_j\neq 0$ for all $1\leq j\leq n.$ Combining it with Examples \ref{ex_special_bar}, a typical function $\alpha$ in $\R^2$ can be for example $\alpha(x_1,x_2):=\frac{11}{12}(\frac{3}{5}x_1^2-\frac{1}{5}x_2^2-\sqrt{2}x_1+\frac{21}{\sqrt{5}}x_2)+\frac{1}{12}x_1x_2.$

\begin{proposition}\label{lem_crit}
	Let $\alpha$ be a special barrier with respect to the convex domain $Y$, with the corresponding special augmenting function $f$ that has no mixed $xy$-terms (i.e., if ${f(x,y)=f_1(x)+f_2(y)}$ for all $x,y\in \R^n$, and for some functions $f_1,f_2$). 
	Then the $\mathcal{C}^{\infty}$ function $\sigma:\R^{2n}\to \R$ defined by  $$\sigma(x,y):=||x||^2-\big(\alpha(x)+f(x,y)\big) ~\text{ for all } x, y\in \R^n,$$ is a Morse function on $\R^n+iY$. Moreover, if we denote $\mathcal{K}_{s}^{\alpha}:=\sigma^{-1}(s)\cap(\R^n+i\overline{Y})$ for all $s\in \R$, then the compact fence $K_{\alpha}=\cup_{s\leq 0}\mathcal{K}_{s}^{\alpha}$, such that each $\mathcal{K}_{s}^{\alpha}\cap(\R^n+iY)$ is a real $\mathcal{C}^{\infty}$ hypersurface except possibly at the set of isolated singularities. 
\end{proposition}

\begin{proof}
By assumption, there is $F_{\alpha}$ such that  $\textsf{Re}F_{\alpha}(x+iy)=\alpha(x)+f_1(x)+f_2(y)$ for all $x,y \in \R^n$. Due to $\alpha +f$ being real part of holomorphic function $F_{\alpha}$, the Laplacian  $\frac{\partial^2(\textsf{Re}F_{\alpha})}{\partial z_j\partial \overline{z_k}}(z)=0$, for all $1\leq j,k\leq n.$ Considering  the Wirtinger operators\footnote{$\frac{\partial}{\partial z_j}:=\frac{1}{2}(\frac{\partial}{\partial x_j}-i\frac{\partial}{\partial y_j})$ and $\frac{\partial}{\partial \overline{z_j}}:=\frac{1}{2}(\frac{\partial}{\partial x_j}+i\frac{\partial}{\partial y_j})$ for $1\leq j,k\leq n.$} followed by taking the real part, this gives
$$\frac{1}{4}\Big(\frac{\partial^2(\textsf{Re}F_{\alpha})}{\partial x_j\partial {x_k}}+\frac{\partial^2(\textsf{Re}F_{\alpha})}{\partial y_j\partial {y_k}}\Big)(x,y)=0 \text{ for all } 1\leq j,k\leq n.$$
Thus for $1\leq j,k\leq n$, at a critical point $p:=(p_1,p_2)\in \R^n+iY$ of $\sigma$, it follows 
\begin{equation}\label{wirti}
	-\frac{\partial^2(\textsf{Re}F_{\alpha})}{\partial x_j\partial {x_k}}(p)=\frac{\partial^2(\textsf{Re}F_{\alpha})}{\partial y_j\partial {y_k}}(p). 
\end{equation}

Next, we need to show that the following block matrix	
\begin{equation}\label{non_deg_matrix}	\renewcommand{\arraystretch}{2.0}
	\nabla:=\begin{pmatrix}	
	\frac{\partial^2\sigma}{\partial x_j\partial x_k}(p)
	& \vline & \frac{\partial^2\sigma}{\partial x_j\partial y_k}(p) \\ 
		\hline
	 \frac{\partial^2\sigma}{\partial y_j\partial x_k}(p)  & \vline & \frac{\partial^2\sigma}{\partial y_j\partial y_k}(p)
	\end{pmatrix}_{1\leq j,k\leq n}= \begin{pmatrix}	
	A
	& \vline & B \\ 
	\hline
	C  & \vline & D
\end{pmatrix} \text{(say)}
\renewcommand{\arraystretch}{1.0}
\end{equation}
	has full rank. Equation (\ref{wirti}), it implies $D=-A$. Moreover, for $1\leq j,k\leq n$, the following are true.

$$\frac{\partial^2\sigma}{\partial x_j\partial {y_k}}(p)
=\Big(\frac{\partial}{\partial x_j}\big(-\frac{\partial f_2}{\partial y_k}\big)\Big)(p)=0,$$
and 
$$\frac{\partial^2\sigma}{\partial y_j\partial {x_k}}(p)
=\Big(\frac{\partial}{\partial y_j}\big(2x_k-\frac{\partial (\alpha+f_1)}{\partial x_k}\big)\Big)(p)=0.$$

 Due to that $\alpha$ is a special barrier with respect to $Y,$ it implies $A$ is positive definite. Thus, $\det(\nabla)=\det(A)\det(D-CA^{-1}B)=(-1)^n \det(A)^2\neq 0.$ Hence by the definition mentioned in Section \ref{sec_2}, the function $\sigma$ is Morse on $\R^n+iY$. Finally, by {\cite[Corollary 2.3]{milnor}}, all the non-degenerate critical points are isolated and the result follows for $\mathcal{K}_{s}^{\alpha}\cap(\R^n+iY)$. The identity $K_{\alpha}=\cup_{s\leq 0} \mathcal{K}_{s}^{\alpha}$ is automatic from the definition of $\mathcal{K}_{s}^{\alpha}.$
\end{proof}


\section{Envelope of special domains in higher complex dimension}\label{sec_5}
In this section, we state and prove a theorem concerning the schlichtness of special domains in $\C^n~(n\geq 2)$.  Before we do so, we establish the following variant of the maximum-modulus principle. Although the result follows as a consequence of Theorem \ref{gun-ross}, it is nevertheless of independent interest. 
Therefore for the reader's convenience, we include the following statement and a brief proof of it.
	
\begin{lemma}\label{key_max_lemma_var}
	Let $D\subset \C^n~(n\geq 2)$ be a pseudoconvex domain and $V\subset D$ be a closed analytic variety of pure dimension at least one. Then for any domain $G\Subset V$ and $f\in \mathcal{O}(D)$ (or $f\in \mathcal{O}(G)$) we have
	$$\max_{\overline{G}}|f|=\max_{bG}|f|.$$
\end{lemma}
\begin{proof}
First consider the case when $f\in \mathcal{O}(D)$.
Since $\overline{G}$ is compact, there exists $z_0\in \overline{G}$ when $|f|\big|_{\overline{G}}$ attains a maximum. If $z_0\in bG$ the result follows. Otherwise, if $z_0\in G,$ choose $D'\subset D$ such that $G=D'\cap V$, i.e. $G=D'\cap V$ is a closed subvariety of $D'$ and $f\in \mathcal{O}(D').$ Since $G$ is a domain (and hence connected), we apply Theorem \ref{gun-ross} to get the desired conclusion.

Next, consider the case when $f\in \mathcal{O}(G)$. Suppose $|f|$ attains maximum at $z_0\in G$. We choose an ambient neighborhood $G'$ of $z_0$ so that $f$ extends to $\widetilde{f}\in \mathcal{O}(G')$ and $V\cap G'\subset G.$ Write $V\cap G'=V_1 \cup V_2\cup \cdots \cup V_k$, where $V_1,\ldots,V_k$ are the irreducible components of $V\cap G'$. By Theorem \ref{gun-ross}, the function $f$ becomes constant on each $V_j$ for $1\leq j\leq k,$ and hence becomes constant on each irreducible components of $G.$ Finally, since $G$ is connected, it follows that $f|_G$ is constant.
\end{proof}

Next, we state and prove one of our main theorems, which is a sharp generalization of Theorem \ref{JP2theorem} (see \cite[Theorem 1]{JP2}) for dimension $n\geq 2$, and also an $n$-dimensional variant of \cite[Theorem 1.2]{haz-por}. Step I of the proof of our theorem is a broader modification of \cite[Example 4.4.(a)]{haz-por} adapted to our general situation, and the remaining steps are done using multiple methods, including Sard's theorem and \cite[Theorem 5.2.11]{stout}. Note that the situation was easier and rather straightforward in {\cite[Example 4.4. (a)]{haz-por}} due to the leaves of the foliations of the level hypersurface mentioned there have the ``boundary" in the horizontal part of the truncated tube, and have only one isolated singularity at the origin. However, our situation is subtle due to the possible involvement of multiple singularities, and also because a special domain $S_{\alpha}$ in general may not be a truncated tube domain. Nevertheless, the maximum modulus principle on the analytic variety ({Lemma \ref{key_max_lemma_var}}) helps us to surmount the difficulty. 

\begin{theorem} \label{1st_theo}
Let $\alpha$ be a special barrier with respect to $Y=B_R(0)\subset \R^n~(n\geq 2)$, and let $f$ be its special augmenting function. Let $X\subset \C^n$ be a convex initial domain, set $D=X+iY$, and let $K_{\alpha}$ be the compact fence given by $\alpha+f$ as defined in (\ref{kalpha}). Then the envelope of holomorphy $\mathcal{E}(S_{\alpha})$ of the special domain ${S_{\alpha}:=D \setminus K_{\alpha}}$ is given by $$\mathcal{E}(S_{\alpha})=D\backslash\widehat{K_{\alpha, \partial}}.$$ Here  $\widehat{K_{\alpha,\partial}}$ denotes the polynomial hull of ${K_{\alpha,\partial}}:=K_{\alpha}\cap (X+ibY)$. In particular, the envelope of holomorphy is schlicht.
\end{theorem}

It is important to note that the statement of the above theorem also indicates that $\widehat{K_{\alpha,\partial}}$ is a subset of $K_\alpha$, which will be proved later in Step I (claim \ref{claim_4}).
The first strategy of our proof is to explicitly determine $\widehat{K_{\alpha,\partial}}$. Note that due to the presence of $\alpha$, our construction becomes subtle, although the difficulties can be surmounted in view of the real-analyticity of $\alpha.$	
\begin{proof} 
\textbf{Step I: Explicit description of $\widehat{K_{\alpha,\partial}}$.} \\

\vspace*{-.3cm} 
 
Given that $K_{\alpha} := \{x+iy \in \R^n+i \overline{Y} : ||x||^2\leq \alpha(x)+f(x,y)\}\subset \C^n.$
 We consider the following holomorphic function $$h(z_1,z_2,\ldots, z_n)=\text{exp}{\Big(\sum_{j=1}^{n}z_j^2-F_{\alpha}(z)\Big)},$$ where $F_{\alpha}$ is the holomorphic function defined in Definition \ref{good_alpha}. Observe that $h$ has modulus $|h|=\text{exp}\big({||x||^2-||y||^2-\big(\alpha(x)+f(x,y)\big)}\big)$.  
 Next for $s\in \R$, we consider the level sets of $|h|$ as follows $$\mathcal{H}_{s}^{\alpha}:=\{z=x+iy\in\C^{n}:||x||^2-||y||^2-\big(\alpha(x)+f(x,y)\big)=s\}.$$ 

Observe that $\mathcal{H}_{s}^{\alpha}$'s are real hypersurfaces of $\C^n$ except possibly at the set of singularities $\textsf{sing}(\mathcal{H}_{s}^{\alpha})$. Moreover the set $\textsf{reg}(\mathcal{H}_{s}^{\alpha})$  is Levi flat for each $s\in \R$. This is because for all $s \in \R$ 
 $$\mathcal{H}_s^{\alpha}=\Big\{z\in\C^{n}: \rho(z,\overline{z}):=\sum_{j=1}^{n}\Big(\frac{z_j+\overline{z_j}}{2}\Big)^2-\sum_{j=1}^{n}\Big(\frac{z_j-\overline{z_j}}{2i}\Big)^2-\textsf{Re}F_{\alpha}(z)=s\Big\}.$$
Due to being real part of holomorphic function the Laplacian  $\frac{\partial^2(\textsf{Re}F_{\alpha})}{\partial z_j\partial \overline{z_k}}(\xi)=0$ with $1\leq j,k\leq n$, and for $X_{\xi}=\sum_{j =1}^{n} a_j \frac{\partial}{\partial z_j}\Big\rvert_{\xi} \in T_{\xi}^{(1,0)}b \big(\textsf{reg}(\mathcal{H}_s^{\alpha})\big)$ the following sesquilinear form $$\mathcal{L}(X_{\xi}, X_{\xi}):=\sum \limits_{j,k =1}^{n} {a_j}\overline{a_k} \frac{\partial^2\rho}{\partial {z_j}\partial \overline{z}_k} (\xi)$$ is zero too. Therefore by Theorem \ref{foli} (\cite{bogg}) the set $\textsf{reg}(\mathcal{H}_s^{\alpha})$
is foliated by a {1-parameter} family of complex submanifolds of complex dimension $(n-1)$. It is nevertheless important to note that for $t\in \R$ and $s\in \R$, there exists the following disjoint complex analytic sets
\begin{equation}\label{comp_ana_sets}
\mathcal{V}_{{s}}^{\alpha,t}=\{z\in \C^n: \sum\limits_{j=1}^{n}z_j^2-F_{\alpha}(z)=s+it\}
\end{equation}
such that $\cup_{t\in \R}{}\mathcal{V}_{{s}}^{\alpha,t}=\mathcal{H}_{{s}}^{\alpha}.$ To verify this, note that any $z_0\in \mathcal{H}_{s}^{\alpha}$ belongs to $\mathcal{V}_{s}^{\alpha,\textsf{Im}\Psi(z_0)}$, where $\Psi(z):=\sum_{j=1}^{n}z_j^2-F_{\alpha}(z).$ The other inclusion is obvious as $\mathcal{V}_{s}^{\alpha,t} \subset \mathcal{H}_{s}^{\alpha},$ for all $t\in \R.$\\

Next, denote the following set $$\mathcal{A}_{\alpha}=\{z=x+iy\in \C^n: ||x||^2-||y||^2-(\alpha(x)+f(x,y))+R^2\leq 0\}.$$

Our next objective is to show that $\mathcal{A}_{\alpha} \cap \overline{D} =\widehat{K_{\alpha,\partial}}.$ 
We present the proof by distributing it among the following three technical claims.

\begin{claim}\label{claim_1}
For every $z_0=x_0+iy_0 \in \mathcal{A}_{\alpha} \cap \overline{D} $ it implies $\mathcal{H}_{s_0}^{\alpha}\cap \big(X+ibY\big)\subset {K_{\alpha, \partial}}$, where $s_0=||x_0||^2-||y_0||^2- (\alpha(x_0)+f(x_0,y_0))$. 
\end{claim}
\begin{claimproof}
Since $z_0 =x_0+iy_0\in \mathcal{A}_{\alpha} \cap \overline{D}$, it follows
\begin{equation}\label{claim_eq_1}
	s_0=||x_0||^2-||y_0||^2-\big(\alpha(x_0)+f(x_0,y_0)\big)\leq -R^2.
\end{equation}
Suppose the claim is not true, and let ${z'=x'+iy' \in \mathcal{H}_{s_0}^{\alpha}\cap (X+ibY)}$ be such that $z'\notin K_{\alpha, \partial}.$ It therefore implies
\begin{equation}\label{claim_eq_2}
||y'||^2=R^2 \text{ and } ||x'||^2>\alpha(x')+f(x',y'),
\end{equation}
Since $z'=x'+iy'\in \mathcal{H}_{s_0}^{\alpha}$, it implies 
\begin{equation}\label{claim_eq_3}
	||x'||^2-||y'||^2-\big(\alpha(x')+f(x',y')\big)=s_0.
\end{equation}
Using equations (\ref{claim_eq_1}), (\ref{claim_eq_2}) and (\ref{claim_eq_3}), it implies, 
$$-R^2=-||y'||^2<s_0\leq -R^2. $$
This leads to a contradiction, and hence $\mathcal{H}_{s_0}^{\alpha}\cap \big(X+ibY\big)\subset K_{\alpha, \partial}$.
\end{claimproof}

\begin{claim}\label{claim_3}
$\mathcal{A}_{\alpha} \cap \overline{D} \subset \widehat{K_{\alpha,\partial}}.$
\end{claim}
\begin{claimproof}
	To prove the inequality $\mathcal{A}_{\alpha} \cap \overline{D} \subset \widehat{K_{\alpha, \partial}}$ we first consider an element $\tilde{x}+i\tilde{y} \in \mathcal{A}_{\alpha} \cap \overline{D}$. Setting $\tilde{s}=||\tilde{x}||^2-||\tilde{y}||^2- (\alpha(\tilde{x})+f(\tilde{x},\tilde{y}))$, it implies $\tilde{z}=\tilde{x}+i\tilde{y}\in\mathcal{H}_{\tilde{s}}^{\alpha}$ and $\tilde{s}\leq -R^2.$ By Claim \ref{claim_1} it therefore implies ${\mathcal{H}_{\tilde{s}}^{\alpha}\cap \big(X+ibY\big)\subset K_{\alpha, \partial}.}$ 
	Moreover, if there is any $\hat{z}=\hat{x}+i\hat{y}\in \mathcal{H}_{\tilde{s}}^{\alpha} \cap (bX+i\overline{Y})$, it implies ${||\hat{x}||^2-\big(\alpha(\hat{x})+f(\hat{x},\hat{y})\big)+(R^2-||\hat{y}||^2)\leq 0.}$ Since $R^2-||\hat{y}||^2\geq 0$, it must imply $||\hat{x}||^2-(\alpha(\hat{x})+f(\hat{x},\hat{y}))\leq 0.$ Consequently, $\hat{z}\in K_{\alpha}\cap (bX+i\overline{Y})$ contradicts the equation (\ref{cvx_initial}). Therefore
	\begin{equation}\label{h_s_and_bdX}
	\mathcal{H}_{\tilde{s}}^{\alpha} \cap (bX+i\overline{Y})=\emptyset.
	\end{equation}
Note that for $t\in \R$, there exists the following disjoint complex analytic sets $\mathcal{V}_{\tilde{s}}^{\alpha,t}$ as described in equation (\ref{comp_ana_sets}), such that $\cup_{t\in \R}{}\mathcal{V}_{\tilde{s}}^{\alpha,t}=\mathcal{H}_{\tilde{s}}^{\alpha}.$  Thus, for some $\tilde{t}\in \R$ it implies $\tilde{z}\in \mathcal{{V}}_{\tilde{s}}^{\alpha,\tilde{t}}$ and $\mathcal{V}_{\tilde{s}}^{\alpha,\tilde{t}}\cap (X+ibY) \subset K_{\alpha, \partial}.$ Moreover, by equation (\ref{h_s_and_bdX}), it follows $\mathcal{V}_{\tilde{s}}^{\alpha,\tilde{t}} \cap (bX+i\overline{Y})=\emptyset.$ Next we apply the maximum modulus principle on the relatively compact domain  $G_{\tilde{s}}^{\alpha,\tilde{t}}:=\mathcal{V}_{\tilde{s}}^{\alpha,\tilde{t}}\cap (X+iY)$ of $\mathcal{V}_{\tilde{s}}^{\alpha,\tilde{t}}$ that contains $\tilde{z}$. Thus using Lemma \ref{key_max_lemma_var} and equation (\ref{h_s_and_bdX}) for any holomorphic polynomial $P:\C^n \to \C$, it implies 
	$$|P(\tilde{z})|\leq \max_{G_{\tilde{s}}^{\alpha,\tilde{t}}}|P|=\max_{bG_{\tilde{s}}^{\alpha,\tilde{t}}}|P|\leq \max_{\mathcal{V}_{\tilde{s}}^{\alpha,\tilde{t}}\cap (X+ibY)}|P|\leq \max_{{K_{\alpha, \partial}}} |P|.$$
	Which implies $\tilde{z}=\tilde{x}+i\tilde{y}\in \widehat{K_{\alpha,\partial}},$ and proves Claim \ref{claim_3}.
\end{claimproof}

\begin{claim}\label{claim_2}
$\widehat{K_{\alpha, \partial}}=\mathcal{A}_{\alpha}\cap\overline{D}.$
\end{claim}

\begin{claimproof}
For the other direction, we consider ${z'\in \overline{D}\setminus (\mathcal{A}_{\alpha} \cap \overline{D})}$, and the holomorphic function $h(z)=\text{exp}{\big(\sum_{j=1}^{n}z_j^2-F_{\alpha}(z)\big)}$ defined earlier. Notice that if any $\zeta=x+iy\in K_{\alpha,\partial}$, then $||x||^2-\big(\alpha(x)+f(x,y)\big)\leq 0$ and $||y||=R$. Due to $z'\notin \mathcal{A}_{\alpha}$, it implies $||x'||^2-||y'||^2-\big(\alpha(x')+f(x',y')\big)>-R^2$. Therefore,
\begin{align*}
	|h(z')| &= \text{exp}\big(||x'||^2-||y'||^2-\big(\alpha(x')+f(x',y')\big)\big)\\
	&>e^{-R^2}= \max_{\zeta\in K_{\alpha,\partial}}\text{exp}\big(||x'||^2-\big(\alpha(x')+f(x',y')\big)-||R||^2\big)\\
	&= \max_{\zeta\in K_{\alpha,\partial}}|h|.
\end{align*}
This shows $\mathcal{A}_{\alpha} \cap \overline{D}= \widehat{K_{\alpha,\partial}}$, and hence the Claim \ref{claim_2} is proved.
\end{claimproof}

The containment of $\widehat{K_{\alpha,\partial}}$ inside $K_{\alpha}$ follows from the  next claim.
\begin{claim}\label{claim_4}
$\widehat{K_{\alpha,\partial}}\subsetneq K_{\alpha}.$
\end{claim}

\begin{claimproof}
Choose any $x^*+iy^* \in \widehat{K_{\alpha,\partial}}$, and notice from Claim \ref{claim_2} that $||x^*||^2-||y^*||^2-\alpha(x^*)-f(x^*,y^*)+R^2\leq 0.$ Since on $\overline{D}$ the quantity $R^2-||y^*||^2\geq 0$ it must therefore imply $||x^*||^2-\alpha(x^*)-f(x^*,y^*)\leq 0.$ This guarantees that $x^*+iy^* \in K_{\alpha}.$ It is worth mentioning that the inclusion is proper can readily be seen from the construction (e.g., choosing a point $z'=x'+iy' \in K_{\alpha}$ satisfying $||x'||^2-\alpha(x')-f(x',y')=0$ and $y'\in B_R(0)$, will be sufficient to conclude ${z'\notin \widehat{K_{\alpha,\partial}}}$). This proves the claim.
\end{claimproof}

Having established the explicit description of $\widehat{K_{\alpha,\partial}}$, we now proceed to discuss the next step.

\noindent \textbf{Step II: Holomorphic extension up to $D\setminus \widehat{K_{\alpha,\partial}}$.}\\

 \vspace*{-.3cm} 

First we note that by convexity, $\widehat{K_{\alpha,\partial}}$ is contained in the compact convex set
\begin{equation}\label{ch_contained}
	\textsf{ch}(\{x\in \R^n: x+iy \in K_{\alpha,\partial} \text{ for some } y\in b{Y} \})+i\overline{Y},
\end{equation}
where $\textsf{ch}$ stands for the convex hull. By the assumption (\ref{cvx_initial}), it implies
\begin{equation}\label{eps_gurrenty}
	{\widehat{K_{\alpha,\partial}} \cap (b X+i\overline{Y}) =\emptyset.}
\end{equation} 
Denote $\Omega_{\alpha}:=D\setminus \widehat{K_{\alpha,\partial}}$. We observe that $D$ is a bounded 
pseudoconvex (infact a convex) domain in $\C^n$, but not of class $\mathcal{C}^2$. But in our situation, we don't need the $\mathcal{C}^2$-boundary (unlike Lemma \ref{cvx-lem}). To prove the holomorphic extension up to $D\setminus \widehat{K_{\alpha,\partial}}$, we use the following well-known theorem.

\begin{theorem}[\cite{st-81}, \cite{Lupa-86}, \cite{chir-st}; Theorem 5.2.11 of \cite{stout}]\label{key_theorem}
Consider $E$ as a compact, $\mathcal{O}(\mathcal{M})$-convex subset of the Stein manifold $\mathcal{M}$, and let $\Omega$ be a bounded domain in $\mathcal{M}$. Assume that both $\Omega\setminus E$ and $\mathcal{M}\setminus (\overline{\Omega}\cup E)$ are connected, and let $ \Gamma:= b\Omega \setminus E$ be a real hypersurface of class $\mathcal{C}^1$.
Then for every continuous CR function $f$ on $\Gamma$, there exists a function $F$ holomorphic on $\Omega \setminus E$ and continuous on $\Gamma \cup (\Omega \setminus E)$, such that $F|_{\Gamma}=f$.
\end{theorem}

For the reader's convenience, we set $\mathcal{M}=\C^n, E= \widehat{K_{\alpha,\partial}}$, and $\Omega=\Omega_{\alpha}$ in the preceding theorem. However, to avoid abuse of an notation, we will keep our original notation of $\Omega_{\alpha}$ throughout. Next, we list our immediate observations as follows.
\begin{itemize}	
\item[(i)] $ \Gamma= b\Omega_{\alpha} \setminus\widehat{K_{\alpha,\partial}}$ is a real hypersurface of class $\mathcal{C}^1$.
 
\item[(ii)] By applying \cite[Theorem 1.3.11]{stout}, we obtain  $\widehat{K_{\alpha,\partial}}=(\widehat{K_{\alpha,\partial}})_{\mathsf{psh}(\C^n)}.$ Next, due to the Stein property of $\C^n$, and using \cite[Proposition 2.7.7]{JP1}, it implies  $$\big(\widehat{K_{\alpha,\partial}}\big)_{\mathsf{psh}(\C^n)}=\big(\widehat{K_{\alpha,\partial}}\big)_{\mathcal{O}(\C^n)}.$$ Consequently, $$\widehat{E}_{\mathcal{O}(\C^n)}=\widehat{E}=\widehat{\widehat{K_{\alpha,\partial}}}=\widehat{K_{\alpha,\partial}}=E.$$
Therefore, $E$ is compact $\mathcal{O}(\C^n)$-convex. \\

\item[(iii)] We see that $\C^n\setminus (\overline{\Omega}_{\alpha}\cup E)= \C^n\setminus \overline{D}$. Due to that $\overline{D}$ is convex and that the complement of a convex set in $\C^n$ is connected, the set $\C^n\setminus \overline{D}$, and hence $\C^n\setminus (\overline{\Omega}_{\alpha}\cup E)$ is connected.
\end{itemize}
We are yet to prove the connectedness of $\Omega_{\alpha}\setminus E$ before we can apply Theorem \ref{key_theorem}. Therefore, we proceed with our next claim.
\begin{claim}\label{claim_ctd}
For $\Omega_{\alpha}$ and $E$ described above, the set $\Omega_{\alpha}\setminus E$ is connected.
\end{claim}
\begin{claimproof}
Since by construction $\Omega_{\alpha}\setminus E=\Omega_{\alpha}$, our claim therefore requires us to prove the connectedness of $\Omega_{\alpha}$. Choose any two distinct points $x_1+iy_1$ and $x_2+iy_2$ in $\Omega_{\alpha}$. As by construction $\widehat{K_{\alpha,\partial}}$ is contained in a compact convex set as in equation (\ref{ch_contained}), it follows from equation (\ref{eps_gurrenty}) that there exists a slightly deformed convex domain $X_{\epsilon}\subset X\subset \R^n$ (see Figure \ref{Salpha}) containing both $x_1$ and $x_2$. Moreover, $X_{\eps}$ can be chosen so that $( X_{\epsilon}+i{Y})\setminus\widehat{K_{\alpha,\partial}} \subset \Omega_{\alpha}$ and 
\begin{equation}\label{deformed_k_cond}
\widehat{K_{\alpha,\partial}} \cap (bX_{\epsilon}+i\overline{Y})=\emptyset.
\end{equation}
Denote $D_{\eps}:=X_{\epsilon}+i{Y}$, and  observe the following.
\begin{itemize}
\item[(i)] For all $y_0\in Y$, the ``{slice}" defined by $S_{y_0}:=\widehat{K_{\alpha,\partial}}\cap (\R^n+i\{y_0\})$ is convex. To verify this, set ${\delta(x,y)=||x||^2-||y||^2-(\alpha(x)+f(x,y))+R^2}$ for all $x,y \in \R^n$,
and note that $$S_{y_0}=\{x\in \R^n: \delta(x,y_0)\leq0\}\cap \overline{D}.$$
Since $\alpha $ is a special barrier, the Hessian matrix $\nabla^2\delta(\textbf{x},y_0)$ of $\delta(x,y_0)$ given by the matrix mentioned in equation (\ref{matrix_cond}), is positive definite for all $\textbf{x} \in X$ (hence for all $\textbf{x}\in X_{\eps}).$  This proves that $\delta(.,y_0)$ is a convex function, and hence its sublevel set $S_{y_0}$ is convex.\\
	
\item[(ii)]  By the previous argument, both the sets $\C^n\setminus S_{y_1}$ and $\C^n\setminus S_{y_2}$ are connected, as both are complements of the convex sets in $\C^n$. Due to equation (\ref{deformed_k_cond}), it implies $S_{y_j}\cap (bX_{\epsilon}+i\overline{Y})=\emptyset$ for all $j=1,2.$  Therefore, there is a path $\gamma_{1}\subset (X_{\eps}+i\{y_1\})\cap \Omega_{\alpha}$ that joins $x_1+iy_1$ and $x_2+iy_1$. \\

\item[(iii)] Since $bX_{\eps}+iY\subset \Omega_{\alpha}$, there exists a continuous curve $\gamma_{2}:[0,1]\to \C^n$ with $\gamma_{2}(\left[0,1\right))\subset ({X_{\eps}}+i\{y_1\})\cap \Omega_{\alpha}$, that joins $x_2+iy_1$ to some point $\gamma_{2}(1)=\zeta_1\in (bX_{\epsilon}+iY)\cap {\Omega_{\alpha}}$. Similarly, there exists a continuous curve $\gamma_{3}:[0,1]\to \C^n$ with $\gamma_{3}(\left[0,1\right))\subset ({X_{\eps}}+i\{y_2\})\cap \Omega_{\alpha}$, joining $x_2+iy_2$ to some point $\gamma_{3}(1)=\zeta_2\in (bX_{\epsilon}+iY)\cap {\Omega_{\alpha}}$.\\

\item [(iv)] Due to equation (\ref{deformed_k_cond}), and that the boundary $bX_{\eps}$ of the convex set $X_{\eps}$ is connected, it implies that the ``{vertical boundary}" $bX_{\epsilon}+iY$ of $D_{\epsilon}$ is connected. Therefore there exists a continuous curve $\gamma_{4}\subset bX_{\epsilon}+iY \subset \Omega_{\alpha}$ joining $\zeta_1$ and $\zeta_2.$

\begin{figure}[!htb]
	\centering
	\vspace*{-5.1cm}
	\def\svgwidth{\linewidth}
	\fontsize{12pt}{1em}
	\scalebox{.9}{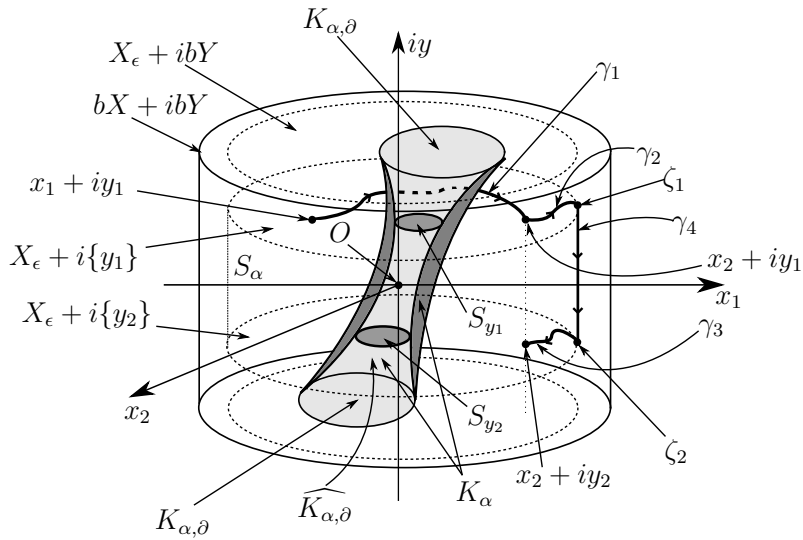}
	\vspace*{-6.0cm}
	\caption{Geometry of the domain $S_{\alpha}$ in $\C^2$}
	\label{Salpha}
\end{figure}
\end{itemize}

Finally, we define the continuous path $\gamma:[0,1]\to \Omega_{\alpha}$ given by the following concatenation $$\gamma:=\gamma_{1}*\gamma_{2}*\gamma_{4}*\gamma^{-1}_{3},$$
which connects $x_1+iy_1$ to $x_2+iy_2$. This completes the proof of Claim \ref{claim_ctd}.
\end{claimproof}\\

Hence by Theorem \ref{key_theorem},  every continuous CR function on $\Gamma =b\Omega_{\alpha}\setminus \widehat{K_{\alpha,\partial}}$ have holomorphic extension up to $\Omega_{\alpha}=D\setminus \widehat{K_{\alpha,\partial}}.$ This completes the proof of Step II.\\

\noindent \textbf{Step III: Pseudoconvexity of $D\setminus \widehat{K_{\alpha,\partial}}$ and the final envelope.}\\

\vspace*{-.3cm} 
 
From step II, recall that \begin{align*}
	\Omega_{\alpha}&:=D\setminus \widehat{K_{\alpha,\partial}}=D \cap \{\delta(x,y)>0\},
\end{align*} 
where as earlier, $\delta(x,y):=||x||^2-||y||^2-(\alpha(x)+f(x,y))+R^2$ for $x,y\in \R^n$. Due to Sard's theorem \cite[Theorem 6.1]{milnor}, the set of all
critical values of the function $\delta$ is of measure zero in $\R$. Therefore, there exists a decreasing sequence $\{r_n\}_{n\in \N}$ of real numbers with $r_n\downarrow 0$ as $n\uparrow \infty$, such that the sets ${\Omega_{\alpha,r_n}:=D\cap \{\delta(x,y)>r_n\}}$ has no critical points at the boundary $b\Omega_{\alpha,r_n}$. Clearly, $\Omega_{\alpha,r_n}$ is an intersection of a domain $\{\delta(x,y)>r_n\}$ with Levi-flat boundary and a pseudoconvex domain $D$. By \cite[Corollary 3.4.7]{kra}, this implies $\Omega_{\alpha,r_n}$ is pseudoconvex. Due to that $\{r_n\}_{n\in \N}$ is decreasing, it implies $\Omega_{\alpha,r_n}\subset \Omega_{\alpha,r_{n+1}}$ for all $r_n\in \R.$ Finally, applying {Behnke-Stein} theorem (see \cite{beh-ste}) on the increasing union of pseudoconvex domains $\Omega_{\alpha,r_n}$, we conclude that $\Omega_{\alpha}=\cup_{\{r_n:n\in \N\}}\Omega_{\alpha,r_n}$ is pseudoconvex. This completes the proof of the theorem.
\end{proof}

From Step II of the above theorem, one can observe that the proof offers more than what is stated in the Theorem \ref{1st_theo}. This means, it yields a holomorphic extension starting with the CR data prescribed on the boundary $\Gamma$ mentioned above. To avoid repetition of the proof, at this point, we therefore formulate the observation as a seperate result.

\begin{theorem}
	Let $\alpha, f, X,Y, S_{\alpha}$ and $K_{\alpha}$ be as in Theorem \ref{1st_theo}. Then every continuous CR function defined on  $\Gamma'=bS_{\alpha}\setminus K_{\alpha}$ has holomorphic extension up to $D\backslash\widehat{K_{\alpha, \partial}},$ where  $\widehat{K_{\alpha,\partial}}$ denotes the polynomial hull of ${K_{\alpha,\partial}}:=K_{\alpha}\cap (X+ibY).$ 
\end{theorem}

As mentioned earlier, our theorem is a sharp generalization of the following well-known result.	
\begin{corollary}[Jarnicki/Pflug, \cite{JP2}]
The envelope of holomorphy of the domain $\{x \in \R^n: r_1<||x||<r_2\}+i\{y\in \R^n:||y||<r_3\}$ is given by $$\{x+iy\in \C^n:||y||^2<||x||^2-(r_1^2-r_3^2)\}.$$ In particular, the envelope is schlicht.
\end{corollary}	
\begin{proof}
Put $\alpha(x)=r_1^2$, $X=\{x\in \R^n: ||x||<r_2^2\}$ and $R=r_3$ in Theorem \ref{1st_theo}.
\end{proof}  
	\begin{corollary}
	Let $r_3,\eps \in \R^+,\delta_j\in \R$, and $C:=\{\sum_{j=1}^{n}(x_j-\frac{\delta_j}{2})^2= \frac{1}{4}{\sum_{j=1}^{n}\delta_j^2}\}$. Assume $r_2>(l+\eps)$, where $l=\max_{x\in C}||x||$.
	Then the envelope of holomorphy of the truncated tube domain $D_0:=\{x \in \R^n: \sum_{j=1}^{n}\delta_jx_j<||x||^2<r_2^2\}+i\{||y||<r_3\}$ is given by $$ \Big\{x+iy\in B_{r_2}(0)+iB_{r_3}(0): ||x||^2-||y||^2>\sum\limits_{j=1}^{n}\delta_jx_j-r_3^2\Big\}.$$ In particular, the envelope is schlicht.
\end{corollary}	
\begin{proof}
	Put $\alpha(x)=\sum_{j=1}^{n}\delta_jx_j$, $X=\{x\in \R^n: ||x||<r_2^2\}$ and $Y=B_{r_3}(0)$ in Theorem \ref{1st_theo}. Clearly $\alpha$ is a special barrier with respect to $Y$ and with zero augmenting function (Example \ref{ex_special_bar}). Thus  ${K_{\alpha} = \{x: ||x||^2\leq \sum_{j=1}^{n}\delta_jx_j\}+i\overline{Y}}$. Denote $K_{\alpha,x}=\{x\in \R^n: x+iy \in K_{\alpha} \text{ for some } y\in \overline{Y}\}.$ Observe that $K_{\alpha,x}$ is the non-shaded region drawn in Figure \ref{ex_circles}, that contains all points bounded by the $(n-1)$-sphere of radius $\sqrt{\frac{1}{4}{\sum_{j=1}^{n}\delta_j^2}}$ and centred at $(\frac{\delta_1}{2},\ldots, \frac{\delta_n}{2})$.	
	\begin{figure}[!htb]
		\centering
		\vspace*{-3.8cm}
		\def\svgwidth{\linewidth}
		\fontsize{11pt}{1em}
		\scalebox{.75}{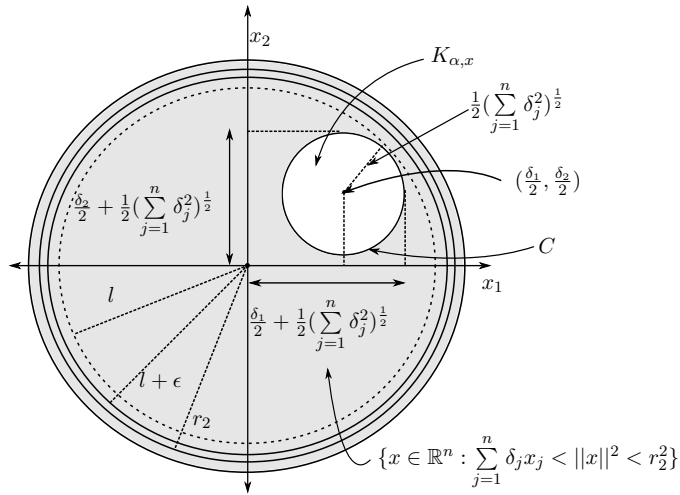}
		\vspace*{-5.4cm}
		\caption{Geometry of the domain $\textsf{Re}(S_{\alpha})$ in $\R^2$}
		\label{ex_circles}
	\end{figure}	
	Observe that the circle $C_{l+\eps/2}$ centred at origin and of radius $(l+\eps/2)$ doesn't intersect $C$. Thus with $r_2>(l+\eps)$ we obtain $ {K_{\alpha} \cap \big(b B_{r_2}(0)+iB_{r_3}(0)\big) =\emptyset}$ (equation (\ref{cvx_initial})), and hence $X$ becomes a convex initial domain. Therefore by Theorem \ref{1st_theo} the envelope of holomorphy of the special domain $(X+iY)\setminus K_{\alpha}=D_0$
	provides the desired result. In particular, the envelope is schlicht.
\end{proof}
	
Before ending this section we note that the final envelope in Theorem \ref{1st_theo} is obtained by deletion of the polynomially convex set $\widehat{K_{\alpha,\partial}}$ from $D.$ In general, a polynomially convex set can possess wild singular structure. However in our situation, with an assumption of $f$ being independent of mixed $xy$-term (as defined earlier in Proposition \ref{lem_crit}), the set $\widehat{K_{\alpha,\partial}}$ can be written as a disjoint union of $\mathcal{C}^{\infty}$ real hypersurface with exceptions at the isolated singularities.

\begin{proposition}\label{lem_crit2}
	Let $\alpha, f, X, Y, \mathcal{H}_{s}^{\alpha},K_{\alpha, \partial}, D $ and $F_{\alpha}$ as in Theorem \ref{1st_theo}. Then it implies 
	\begin{equation}\label{s-R2}
	\widehat{K_{\alpha, \partial}}=\mathcal{A}_{\alpha}\cap \overline{D}=\Big(\cup_{s\leq -R^2}\mathcal{H}_{s}^{\alpha}\Big)\cap \overline{D}.
	\end{equation}
Moreover $f$ is independent of the mixed $xy$-term, then each $\mathcal{H}_{s}^{\alpha} \cap D$ in the union is a real hypersurface of class $\mathcal{C}^{\infty}$ excepts possibly at the set of the isolated singularities.
\end{proposition}

	\begin{proof}
	A simple computation along with Claim \ref{claim_2} shows the equality in equation (\ref{s-R2}). Set $\mu(x,y):=||x||^2-||y||^2-\big(\alpha(x)+f(x,y)\big)$ and $f(x,y)=f_1(x)+f_2(y)$ for all $x,y\in\R^n.$ Note that the non-degeneracy of a critical point $p:=(p_1,p_2)\in D$ of $\mu$ follows from the matrix mentioned in equation (\ref{non_deg_matrix}) replacing $\sigma$ by $\mu$, and using the identity (\ref{wirti}). To verify this, we finally obtain the following identities for $1\leq j,k\leq n$.

$$\frac{\partial^2\mu}{\partial x_j\partial {y_k}}(p_1,p_2)     
=\Big(\frac{\partial}{\partial x_j}\big(-2y_k-\frac{\partial f_2}{\partial y_k}\big)\Big)(p_1,p_2)=0,$$
and 
$$\frac{\partial^2\mu}{\partial y_j\partial {x_k}}(p_1,p_2)
=\Big(\frac{\partial}{\partial y_j}\big(2x_k-\frac{\partial (\alpha+f_1)}{\partial x_k}\big)\Big)(p_1,p_2)=0.$$

Therefore as earlier in Proposition \ref{lem_crit}, the matrix described in equation (\ref{non_deg_matrix}) is non-singular. Finally by \cite[Corollary 2.3]{milnor} all the non-degenerate critical points are isolated and the result follows for $\mathcal{H}_{s}^{\alpha}\cap D$.
\end{proof}


\section{Schlichtness of the envelope via Stout's theorem on analytic continuation}\label{sec_6}	

In this section, we state and prove two theorems concerning the envelope of holomorphy of truncated tube domains. The theorems presented below generalize Theorem \ref{JP2theorem} in $\C^n$ respectively for $n\geq 2$ and $n\geq 3$. 

\begin{theorem}\label{2nd_theo}
	Let $X,Y\Subset \R^n~(n\geq 2)$ be convex domains, and set ${D=X+iY}$. Denote ${H:=H_x+ibY}$, where $H_x\subset X$ is a compact set such that ${(X\setminus H_x)+ibY}$ is connected. Let $E:=b\widehat{H}\setminus\textsf{int}_{\overline{D}}(\widehat{H})$ be given by $$E=\{z\in \overline{D}: \textsf{Re}(F_j(z))=0 \text{ for all }1\leq j\leq m\},$$ for some holomorphic functions $F_j:\C^n \to \C$ with $1\leq j\leq m$. If $\widehat{H}\subset H_x+i\overline{Y}$, then the envelope of holomorphy of the truncated tube domain $\mathcal{T}:=(X\setminus H_x)+iY$ is schlicht, and it is given by $$\mathcal{D}:=\{z\in D: \delta_j\textsf{Re}(F_j(z))>0 \text{ for all } 1\leq j\leq m\},$$
for some choice of $\delta_j\in\{ -1,1\}$ with $j\in \{1,2,\ldots, m\}.$
\end{theorem}

\begin{proof}
Notice that the choice of $\delta_j$'s mentioned above depends on which side $\widehat{H}$ lies. For instance, if $\widehat{H} \subset \{z\in \overline{D}: \textsf{Re}(F_j(z))\leq0 \text{ for all }1\leq j\leq m\}$, then in the definition of $\mathcal{D}$, we take $\delta_j=1$ for all $j\in  \{1,2,\ldots, m\}$. Central idea of the proof relies on an application of Stout's theorem on analytic continuation (Theorem \ref{stout_theo}). Since $E \subset b\mathcal{D}$, the first step is to establish both the pseudoconvexity of $\mathcal{D}$, and  $\mathcal{O}(\overline{\mathcal{D}})$-convexity of $E.$ \\
	\vspace*{-.3cm} 
	
	\noindent\textbf{Step I: Proving $\mathcal{D}$ is pseudoconvex.}\\
	
	\vspace*{-.3cm} 
	
	Due to Sard's theorem \cite[Theorem 6.1]{milnor}, the set of all
	critical values of $\delta_j\textsf{Re}(F_j)$ has Lebesgue measure zero in $\R$ for all $1\leq j\leq m$. Thus for all fixed ${j\in \{1,2,\ldots, m\}}$, there exists a decreasing sequence $\{s_n^j\}_{n\in \N}$ of real numbers with $s_n^j\downarrow 0$ as $n \uparrow \infty$, such that the following sets 
	$$\mathcal{D}_{F_j,s_n^j}:=\{z \in D: \delta_j\textsf{Re}(F_j(z))>s_n^j\}$$ has no critical points at $b \mathcal{D}_{F_j,s_n^j}$. Since $\mathcal{D}_{F_j,s_n^j}$ is an intersection of the domain ${\{z \in \C^n: \delta_j\textsf{Re}(F_j(z))>s_n^j\}}$ with Levi-flat boundary\footnote{The Laplacian is zero due to being the real part of a holomorphic function $F_j$.} and the pseudoconvex domain $D$, it is pseudoconvex by \cite[Corollary 3.4.7]{kra}. Moreover, as ${\mathcal{D}_{F_j,s_n^j}\subset \mathcal{D}_{F_j,s_{n+1}^j}}$ for all $n\in \N,$ we apply Behnke-Stein theorem (see \cite{beh-ste}) to the increasing union of pseudoconvex domains $\mathcal{D}_{F_j,s_n^j}$ to conclude that for each ${j\in \{1,2,\ldots,m\}}$ the following domain $$\mathcal{D}_{F_j}:=\cup_{\{s_n^j:n\in \N\}}\mathcal{D}_{F_j,s_n^j}=\{z \in D: \delta_j\textsf{Re}(F_j(z))>0\}$$ is pseudoconvex. Finally, the pseudoconvexity of the domain $\mathcal{D}= \cap_{j=1}^m \mathcal{D}_{F_j}$ follows from {\cite[Corollary 3.4.7]{kra}}.\\

	\noindent	\textbf{Step II:  Showing $E$ is $\mathcal{O}(\overline{\mathcal{D}})$-convex.}\\
	
	\vspace*{-.3cm} 
	
	For a given $p=x_0+iy_0\in \overline{\mathcal{D}}\setminus E$, there exists a ${k\in \{1,2,\ldots,m\}}$ such that $\delta_k\textsf{Re}(F_k(p))=\eps>0$.  
	Define the entire function ${h_{F_k}(z):=\frac{\exp(\delta_kF_k(z))}{\exp(\epsilon)}}$, and notice that  $$||h_{F_k}||_{E}:=\sup_{z\in E}|h_{F_k}(z)|= \frac{1}{\exp(\epsilon)}<1.$$
Since $|h_{F_k}(p)|=1$, it follows $h_{F_k}(p)=\exp(i\theta')$ for some $\theta'\in \R$. Notice that the  composite function $\Phi_{F_k}:\C^n \to \C$ defined by $\Phi_{F_k}(z):=\exp(-i\theta')h_{F_k}(z)$ for $z\in \C^n,$ is entire. Moreover, $$\Phi_{F_k}(p):=\exp(-i\theta')h_{F_k}(p)=\exp(-i\theta')\exp(i\theta')=1.$$ As the modulus is invariant under the unimodular rotation map, it implies $$||\Phi_{F_k}||_{E}=\sup\limits_{z\in E}|\Phi_{F_k}(z)|=\sup\limits_{z\in E}|h_{F_k}(z)|<1.$$ Hence $E$ is  $\mathcal{O}(\overline{\mathcal{D}})$-convex by Definition \ref{ODRD-cvx}.\\

	\begin{figure}[!htb]
		\centering
		\vspace*{-4.6cm}
		\def\svgwidth{\linewidth}
		\fontsize{12pt}{1em}
		\scalebox{.75}{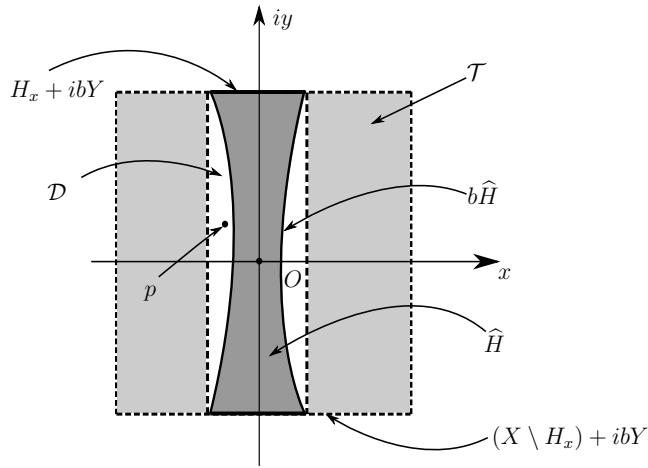}
		\vspace*{-4.9cm}
		\caption{Geometry of the domain $\mathcal{D}\subset \C^n$ in Theorem \ref{2nd_theo}.}
		\label{reF_theo}
	\end{figure}

	\noindent \textbf{Step III: Proving $\mathcal{E}(\mathcal{T})=\mathcal{D}$.}\\ 
	
	\vspace*{-.3cm} 
	
	Note that due to $X\setminus H_x$ being an image of a continuous (``projection") map from the connected set $(X\setminus H_x)+ibY$, it is connected. The connectedness of ${\mathcal{T}={(X\setminus H_x)+iY}}$ therefore follows due to $Y$ being convex (and hence connected). Next the condition $\widehat{H}\subset H_x+i\overline{Y}$ suggests that 
	$\mathcal{T}=(X\setminus H_x)+iY$ is an open set such that $\mathcal{T} \cup (b\mathcal{D}\setminus E)$ is a
	neighborhood of $ b \mathcal{D}\setminus E$ in $\overline{\mathcal{D}}$ (see Figure \ref{reF_theo}). Therefore by Theorem \ref{stout_theo} we conclude that the envelope of holomorphy $\mathcal{E}(\mathcal{T})=\mathcal{D}$. In particular, the envelope is schlicht.
\end{proof}

\begin{remark}
	One could obtain Theorem \ref{2nd_theo} using \cite[Theorem 1]{por-thie} as well. Note that $bD\setminus H$ would become connected under our hypothesis. To verify this, note that 
	\begin{equation}\label{bD_minus_H}
		bD\setminus H=(bX+i\overline{Y})\cup ((X\setminus H_x)+ibY).
	\end{equation}
Due to the convexity of $X,Y$ together with our hypothesis, both sets on the right-hand side are connected. First to prove the CR($bD$)-convexity of $H$, suppose there exists a point $p_1+ip_2\in \widehat{H}\cap (bD\setminus H)$.  Since by hypothesis $\widehat{H}\subset H_x+i\overline{Y}$, it implies $p_1+ip_2\in (H_x+i\overline{Y})\cap (bD\setminus H).$ Using equation (\ref{bD_minus_H}), it follows either $p_1\in H_x\cap bX$ or $p_1\in H_x\cap (X\setminus H_x)$, neither of which can occur. Since $\widehat{H} \supset \widehat{H}_{\mathcal{A}(D)}$, it follows that 
\begin{equation}\label{cr_cond_D}
\widehat{H}_{\mathcal{A}(D)} \cap bD =H.
\end{equation}
Thus $H$ is CR($bD$)-convex. To apply the aforementioned theorem, we can arrange a slightly deformed pseudoconvex domain $D_{\eps}\subset D$ with $\mathcal{C}^2$-boundary  satisfying $bD_{\eps}\cap D \subset \mathcal{T}$, such that $bD_{\eps}\setminus H$ remains connected (see Lemma \ref{cvx-lem}). Since $\widehat{H}_{\mathcal{A}(D_{\eps})}\subset \widehat{H}_{\mathcal{A}(D)}$, it follows from equation (\ref{cr_cond_D}) that $\widehat{H}_{\mathcal{A}(D_{\eps})} \cap bD_{\eps} =H,$ and hence  $H$ is CR($bD_{\eps}$)-convex. If  ${\widehat{H} \subset \{z\in \overline{D}: \delta_j\textsf{Re}(F_j(z))\leq0 \text{ for all }1\leq j\leq m\}}$ for some choice of $\delta_j\in\{ -1,1\}$ with $j\in \{1,2,\ldots, m\},$ then it follows $\mathcal{D}\subset D\setminus\widehat{H}\subset D\setminus \widehat{H}_{\mathcal{A}(D_{\eps})}.$ Since by the aforementioned theorem, every function holomorphic on $\mathcal{T}$ has holomorphic extension up to $D\setminus \widehat{H}_{\mathcal{A}(D_{\eps})},$ and that $\mathcal{D}$ is pseudoconvex (see Step I), the envelope $\mathcal{E}({\mathcal{T}})=\mathcal{D}=D\setminus \widehat{H}$.
\end{remark}
It is important to note that one could replace the hypothesis $\widehat{H}$ by $\widehat{H}_{\mathcal{R}}$ in Theorem \ref{2nd_theo} to obtain schlicht envelope. The proof of the theorem remains the same as earlier even when $b\widehat{H}_{\mathcal{R}} \setminus\textsf{int}_{\overline{D}}(\widehat{H}_{\mathcal{R}})$ is written as a common zero locus of finitely many holomorphic polynomials. However as a special case we address this in an easy corollary of Theorem \ref{2nd_theo} obtained by putting $m=1$.
\begin{corollary}
	Let $X,Y$ be two bounded convex domains in $\R^n$ with $n\geq 2$, and set $D=X+iY$. Denote $H=H_x+ibY$, where $H_x\subset X$ be any compact set such that $X\setminus H_x$ is connected. If $$E:=b\widehat{H}_{\mathcal{R}} \setminus\textsf{int}_{\overline{D}}(\widehat{H}_{\mathcal{R}})=\{z\in \overline{D}: \textsf{Re}(F(z))=0\}$$ for some holomorphic function $F:\C^n \to \C$, and if $\widehat{H}_{\mathcal{R}}\subset H_x+i\overline{Y},$ then the envelope of holomorphy of the truncated tube domain $\mathcal
	T_{F}:=(X\setminus H_x)+iY$ is schlicht, and is given by $$\{z\in X+iY:\delta\textsf{Re}(F(z))>0\},$$ for some $\delta\in \{-1,1\}$ depending upon the orientation of $E.$
\end{corollary}

Next we have two simple corollaries with concise proofs, the details of which are not too technical.
\begin{corollary}
Let $X,Y$ be two bounded convex domains in $\R^n$ with $n\geq 2$, and set $D=X+iY$. Denote $H_x:=\{x\in \R^n: \alpha(x)+f(x)\leq 0\}$ for some special barrier $\alpha$ and its special augmenting function $f$ that doesn't depend on $y$. Then the envelope of holomorphy of the truncated tube domain $(X\setminus H_x)+iY$ is schlicht, and is given by $$\{x+iy\in D: \alpha(x)+f(x)>0\}.$$ 
\end{corollary}
\begin{proof}
Denote $H=H_x+ibY$, and note that  there exists a $F_{\alpha}\in \mathcal{O}(\C^n)$ is such that $\textsf{Re}(F_{\alpha}(x+iy))=\alpha(x)+f(x)$ for all $x,y\in \R^n.$ Moreover, from our earlier discussions in Section \ref{sec_5}, it follows $\widehat{H}=\{z\in \overline{D}: \textsf{Re}(F_{\alpha}(z))\leq0\}.$  
 Note that if $\textsf{Re}(F)$ depends only on $x$, then using Cauchy-Riemann equations one can show that $F$ is affine $\R$-linear\footnote{i.e. there exists $(c_1,\ldots, c_n)\in \R^n$ and $d\in \R$ such that $f(z_1,\ldots,z_n)=c_1z_1+\ldots+c_nz_n+d$.}. This means the real Hessian of $\textsf{Re}(F)$ is identically zero everywhere, proving the convexity of $H_x$ (and hence both $\widehat{H}\subset H_x+i\overline{Y}$ and the connectedness of $X\setminus H_x$). Hence by Theorem \ref{2nd_theo} the result follows. 
\end{proof}

\begin{corollary}
Let $r_1, r_2, r_3\in \R^+$, $r_1<r_2$, and $h:\C^n \to \C$ be a holomorphic function such that $\textsf{Re}(h(x+iy))=h_1(x)-h_2(y)$ for all $x+iy\in \C^n$ and some functions $h_j:\R^n\to \R^+\cup \{0\}$ with $j=1,2$. Let ${X:=\{x\in \R^n: h_1(x)<r_2^2\}}$, $H_x:=\{x\in \R^n: h_1(x)\leq r_1^2\}$, and ${Y:=\{y\in \R^n: h_2(y)<r_3^2\}}$ be bounded and convex, and set $D=X+iY$. If $r_1^2-r_3^2$ is a regular value of $\textsf{Re}(h)$, then the envelope of holomorphy of ${(X\setminus H_x)+iY}$ is schlicht, and is given by $$\{x+iy\in D: h_1(x)-h_2(y)> r_1^2-r_3^2\}.$$
\end{corollary}

\begin{proof}
Observe that in this case the polynomial hull $\widehat{H}$ of $H:=H_{x}+ibY$ is given by $A:=\{x+iy\in \overline{D}: h_1(x)-h_2(y)\leq r_1^2-r_3^2\}$, where $D=X+iY$. To verify this, as earlier note that every point $z_0=x_0+iy_0 \in A$ there is the following level set $$\mathcal{H}_{\textsf{Re}(h(z_0))}:=\{z\in \C^n: \textsf{Re}(h(z))=\textsf{Re}(h(z_0))\}$$ containing $z_0$.  
Since $bY\subset h_2^{-1}(r_3^2)$, by similar techniques used in Claim \ref{claim_1}, one can show that $\mathcal{H}_{\textsf{Re}(h(z_0))} \cap (X+ibY)\subset (H_{x}+ibY).$ By maximum modulus principle (see e.g. Lemma [\ref{key_max_lemma_var}, \ref{gun-ross}]) this yields ${A\subset \widehat{H}}.$ Other inclusion is straightforward due to the entire holomorphic function ${e^{h(z)}}/{e^{r_1^2-r_3^2}}$ which seperates points outside of $A$ from $H$. Since $r_1^2-r_3^2$ is a regular value of $\textsf{Re}(h)$, it implies 
$$E:=b\widehat{H}\setminus\textsf{int}_{\overline{D}}(\widehat{H}) =\{x+iy\in \overline{D}: \textsf{Re}\big(h(z)-r_1^2-r_3^2\big)=0\}.$$ Therefore the result follows due to Theorem \ref{2nd_theo} with the choice of $\delta=1.$
\end{proof}

Following Theorem \ref{haz-por_sch_2}, it is customary to ask about a sufficient condition in dimension three or higher which guarantees the schlichtness of the the envelope of holomorphy of truncated tube domains. As stated earlier in [\textbf{(iv)}, page \pageref{part:iv}] of Section \ref{sec_1}, we next establish the following theorem.

\begin{theorem}\label{3rd_theo}
Let $X$ and $Y$ be any two bounded convex domains in $\R^n, ~{n\geq 3}$.
Denote $K=K_x+ibY$, where $K_x\subset X$ is a compact convex set. Let the following two conditions are satisfied.
	
\begin{itemize}
	\item [(i)] The set $b\widehat{K}_{\mathcal{R}}\cap D$ is a Levi-flat hypersurface of class $\mathcal{C}^2$ in $\C^n$.
	
	\item [(ii)] Every point $p\in N:=\textsf{int}_{\overline{D}}(\widehat{K}_{\mathcal{R}})$  
	lies within a purely $(n-1)$-dimensional analytic variety $V_p$ that doesn't intersect the set ${E:=b\widehat{K}_{\mathcal{R}}\setminus N}$.
\end{itemize}	
Then the envelope of holomorphy of the truncated tube domain $T:=(X\setminus K_x)+iY$ is given by $D\setminus\widehat{K}_{\mathcal{R}},$where $D=X+iY.$ In particular, the envelope is schlicht.
\end{theorem}

\begin{proof}
As earlier, the central idea of the proof follows from Stout's theorem on analytic continuation with the only difference of applying the idea that relies on the rational convexity of $E$.
Denote $\Omega=D\setminus \widehat{K}_{\mathcal{R}}$, and note that $\Omega$ is pseudoconvex. This is due to that the boundary $b \Omega= (b\widehat{K}_{\mathcal{R}}\setminus N) \cup (b D \setminus \widehat{K}_{\mathcal{R}})$, and that the boundaries mentioned in the first and second sets in the union are respectively Levi-flat and pseudoconvex. Therefore $\Omega$ is a bounded domain of holomorphy in $\C^n.$ Next consider the following claim which suggests that our assumptions not only ensure the $\mathcal{R}(\overline{D})$-convexity of $E$ for all dimensions $n\geq 2$, but also confirm the convexity with respect to $(n-1)$-dimensional varieties (Definition \ref{var_cvx}). 
	
	\begin{claim}\label{claim_var_dim_n-1}
	$E$ is convex with respect to varieties of dimension $(n-1).$
	\end{claim}
	
	\begin{claimproof}
The set $$\C^n\setminus E= (\overline{D}\setminus \widehat{K}_{\mathcal{R}}) \cup (\C^n\setminus \overline{D})\cup N.$$ We need to show that any point $p\in \C^n\setminus E$ lies within a $(n-1)$-dimensional variety $V_p$ that doesn't intersect $E.$ Due to hypothesis (ii), we only need to show the argument for the remaining two cases. \\

  \noindent\textbf{(i) Points in  $\overline{D}\setminus \widehat{K}_{\mathcal{R}}$:} 
  Due to the rational convexity of $\widehat{K}_{\mathcal{R}}$, every point $p\in  \overline{D}\setminus \widehat{K}_{\mathcal{R}}$ lies on an analytic variety $\mathcal{V}_p$ (zero locus of some holomorphic polynomial) of pure dimension $(n-1)$ 
 such that $\mathcal{V}_p\cap K=\emptyset$. It is even true that $\mathcal{V}_p\cap \widehat{K}_{\mathcal{R}}=\emptyset.$ Otherwise if some $z'\in \mathcal{V}_p\cap \widehat{K}_{\mathcal{R}}$, then due to $\mathcal{V}_p\cap K=\emptyset$, it yields (by definition of rational convexity) that $z'\notin \widehat{K}_{\mathcal{R}}$, a contradiction. Since $E\subset \widehat{K}_{\mathcal{R}}$ it implies $\mathcal{V}_p\cap E=\emptyset$. \\

 \noindent\textbf{(ii) Points in  $\C^n \setminus \overline{D}$ :}  Lets assume $p=x_0+iy_0\in \C^n\setminus \overline{D}=\C^n\setminus (\overline{X}+i\overline{Y}).$ Then $p\in \big((\R^n\setminus \overline{X}) +i\R^n \big)\cup \big(\R^n+i(\R^n\setminus \overline{Y})\big).$ Suppose $p\in (\R^n\setminus \overline{X}) +i\R^n$. Note that due to $\overline{X}$ being convex and compact, and $x_0\notin \overline{X}$, by Hahn-Banach seperation theorem (see \cite[Theorem 3.7]{rob}) for any point $x_0\notin \overline{X}$ there is a real hyperplane that strongly seperates $\overline{X}$ from $x_0$. In other words, there is a linear functional $\ell:\R^n \to \R$ such that $\ell(x_0)=0$ and $\ell(x)<\ell(x_0)-\epsilon$ for all $x\in \overline{X}$, and for some $\eps>0$. Due to Riesz's representation theorem (\cite[p. 351, Th 13.32]{roman}) for any continuous linear functional $\ell$ on a Hilbert space $\R^n$ there is a unique vector $v_{\ell}\in \R^n$ such that $\ell(x)=\langle x,v_{\ell}\rangle$ for all $x\in \R^n.$ Next consider the $(n-1)$-dimensional variety $V_p^h:=\{z\in \C^n: h(z)=0\}$, which is the zero locus of the following entire function $$h(z):= e^{\langle z-z_0, v_{\ell}\rangle}-1 \text{ for all } z\in \C^n.$$
  Clearly $h(z_0)=0$, and thus $z_0\in V_p^h$.  Moreover for all ${z=x+iy \in \overline{X}+i\R^n}$ it implies $$|h(z)|\geq \big||e^{\langle z-z_0, v_{\ell}\rangle}|-1\big| =\big|e^{\langle x-x_0, v_{\ell}\rangle}-1\big|= \big|e^{\ell(x)-\ell(x_0)}-1\big|>0,$$
 due to that the quantity $e^{\ell(x)-\ell(x_0)}<e^{-\eps}<1.$
  This shows $V_p^h\cap \overline{D}=\emptyset$ in this case, and hence $V_p^h\cap E=\emptyset$. The case when $p\in \R^n+i(\R^n\setminus \overline{Y})$ follows similarly by considering $\tilde{\ell}, v_{\tilde{\ell}}$ satisfying $\tilde{\ell}(y_0)=0$ and $\tilde{\ell}(y)<\tilde{\ell}(y_0)-\epsilon$ for all $y\in \overline{Y}$. Moreover, $\tilde{\ell}(y)=\langle y,v_{\tilde{\ell}}\rangle$  for all $y\in \R^n$. Finally, choose the zero locus of the entire holomorphic function $H(z):= e^{\langle -iz+iz_0,v_{\tilde{\ell}}\rangle}-1 \text{ for all } z\in \C^n$ to get the desired conclusion.\\
  
  Hence by Definition \ref{var_cvx} the proof of Claim \ref{claim_var_dim_n-1} follows.
\end{claimproof}

\vspace{.5cm}
Next note that in $\C^n$ the notion of convexity with respect to varieties
of dimension $(n-1)$ coincides with the notion of rational convexity (see \cite[Page 164]{stout}). Observe that $T$ becomes an open set such that $T \cup (b\Omega \setminus E)$ forms a neighborhood of $ b\Omega\setminus E$ in $\overline{\Omega}$, and that $E$ is $\mathcal{R}(\overline{\Omega})$-convex since it is rationally convex.  Hence by Theorem \ref{stout_theo},  the result follows for $n\geq 3$.
\end{proof}

As an application to this theorem we obtain Theorem \ref{JP2theorem} for $n\geq 3$.
\begin{corollary}[Jarnicki/Pflug, \cite{JP2}]\label{jp_corollary_2norm}
	In $\C^n$ with $n\geq 3$, let $r_1,r_2, r_3$ be real numbers such that $r_1<r_2$ and $r_1\neq r_3$; then the envelope of holomorphy of the domain $$\{x \in \R^n: r_1<||x||<r_2\}+i\{y\in \R^n:||y||<r_3\}$$ is given by $$\{z=x+iy\in \C^n: ||x||<r_2, ||y||<r_3,  ||y||^2<||x||^2-(r_1^2-r_3^2)\}.$$ In particular, the envelope is schlicht.
\end{corollary}	
\begin{proof}
Put $X=\{x\in \R^n: ||x||<r_2\}$, $Y=\{y\in \R^n: ||y||<r_3\}$, $D=X+iY$, and the set ${K_x=\{x\in \R^n: ||x||\leq r_1\}}$ in Theorem \ref{3rd_theo}. 
Due to Proposition \ref{both_hull_same} it implies $\widehat{K}=\widehat{K}_{\mathcal{R}}.$ Therefore $E:=b\widehat{K}_{\mathcal{R}}\setminus \textsf{int}_{\overline{D}}(\widehat{K}_{\mathcal{R}})$ is given by
$$\{x+iy\in \overline{D}:||x||^2-||y||^2= (r_1^2-r_3^2)\}\setminus \{x+iy\in \overline{D}:||x||<r_1, ||y||=r_3\}.$$ The only possible  singularity of the function ${||x||^2-||y||^2-(r_1^2-r_3^2)}$ is at the origin which doesn't lie on $b\widehat{K}_{\mathcal{R}}\cap D$ since $r_1- r_3\neq0$. This proves that the set $b\widehat{K}_{\mathcal{R}}\cap D$ is a Levi-flat hypersurface of class $\mathcal{C}^2$ in $\C^n$, and hence yields condition (i) of Theorem \ref{3rd_theo}. For condition (ii), note that every $p=a+ib\in N:= \textsf{int}_{\overline{D}}(\widehat{K}_{\mathcal{R}})$ lies on the $(n-1)$-dimensional analytic variety $$V_{p}:=\Big\{\sum_{j=1}^{n}z_j^2=(||a||^2-||b||^2)+i\textsf{Im}(F(p))\Big\},$$ where $F(z)=\sum_{j=1}^{n}z_j^2$ (see \cite[Example 4.4.(a)]{haz-por}, or [Step I, Theorem \ref{1st_theo}] for details). If there is any $\tilde{z}=\tilde{x}+i\tilde{y}\in V_{p}\cap E$, then due to $p=a+ib\in N$, it implies
$$r_1^2-r_3^2=||\tilde{x}||^2-||\tilde{y}||^2=||a||^2-||b||^2<r_1^2-r_3^2.$$Therefore it yields $V_{p}\cap E=\emptyset,$ and the result follows.
\end{proof}


\section{Open questions}\label{sec_7}

In conclusion, the results proved earlier encourage us to list some questions for future research, which remain open to the best of our knowledge. Throughout, we assume $n\geq 2,$ unless specified otherwise.

\begin{ques}\label{open_qs_1}
Let $S_{\alpha}=(X+iY)\setminus K_{\alpha}$ be the special domain with $Y\subset \R^n$ be any convex domain as in Definition \ref{special_domain}. 
\begin{itemize}
\item [(i)] What sufficient conditions on $Y$ ensure the schlicht envelope of holomorphy $\mathcal{E}(S_{\alpha})$ of
$S_{\alpha}$?
\vspace{0.2cm}

\item [(ii)] What is the description of $\mathcal{E}(S_{\alpha})$ in this case?

\vspace{0.2cm}
\item[(iii)] Is there an example of a good barrier $\alpha$ that is not a special barrier, such that Theorem \ref{1st_theo} fails?
\end{itemize}

\end{ques}
From our discussions earlier, note that parts (i) and (ii) of Question \ref{open_qs_1} already have positive answers when $Y=B_R(0)$ (see Theorem \ref{1st_theo}). It might be interesting to see whether a construction analogous to \cite[Example 4.3]{haz-por} could be used to solve part (iii).

\begin{ques}[see Remark 3 in \cite{JP2}]
Can a result analogous to Corollary \ref{jp_corollary_2norm} be proved for the truncated tube domain $$\{x \in \R^n: r_1<||x||_4<r_2\}+i\{y\in \R^n:||y||_4<r_3\}?$$
\end{ques}
It might be interesting to understand the situation of Theorem \ref{3rd_theo} for $n=2$ so that it also generalize Theorem \ref{JP2theorem}. Moreover, one may ask whether Theorem \ref{3rd_theo} still remains valid when the set $b\widehat{K}_{\mathcal{R}}\cap D$ is not Levi-flat hypersurface of class $\mathcal{C}^2$ in $\C^n$.

\begin{ques}\label{open_qs_3}
Let $X$ and $Y$ be two convex domains in $\R^n$, and let $K\subset \C^n$ be compact. Under what geometric condition on $K$ does the envelope of holomorphy $\mathcal{E}((X+iY)\setminus K)$ become schlicht?
\end{ques}
Question \ref{open_qs_3} has a positive answer in $\C^2$ (Theorem \ref{haz-por_sch_2}),
when $K=K_x+i\overline{Y}$, where $K_x\subset X\subset \R^2$ is strictly convex (see also Theorem \ref{schlicht2d}). It might be interesting to study the same question under a few additional assumptions, for instance, when $K\cap (bX+i\overline{Y})=\emptyset$ or $K\cap (\R^n+i\{y\})$ is convex for all $y\in \overline{Y}.$\\

\textbf{Acknowledgements:} The author would sincerely like to thank Professor Egmont Porten
for several valuable remarks, suggestions, and fruitful discussions that greatly helped to ameliorate the quality of the article. The author would also like to acknowledge anonymous referees for the useful comments and suggestions provided.


\begin{thebibliography}{45}

\bibitem{beh-ste} H. Behnke and K. Stein: \textit{Konvergente Folgen von Regularitätsbereichen und die Meromorphiekonvexität}, Math. Ann. \textbf{116}~(1939), 204--216. DOI 10.1007/BF01597355. MR\href{https://mathscinet.ams.org/mathscinet/article?mr=1513225}{1513225}.

	
\bibitem{boch} S.~Bochner: \textit{A theorem on analytic continuation of functions in several variables},
Ann. of Math. {\bf 39}~ (1938), 14--19. DOI 10.2307/1968709. MR\href{https://mathscinet.ams.org/mathscinet/article?mr=1503384}{1503384}.

\bibitem{bogg} A. Boggess, {\it CR manifolds and the tangential Cauchy-Riemann complex}, Studies in Advanced Mathematics, CRC, Boca Raton, FL, 1991; MR\href{https://mathscinet.ams.org/mathscinet/article?mr=1211412}{1211412}.
	
\bibitem{chir-st} E.~M. Chirka and E.~L. Stout, \textit{Removable singularities in the boundary}, in {\it Contributions to complex analysis and analytic geometry}, 43--104, Aspects Math., E26, Friedr. Vieweg, Braunschweig, ; MR\href{https://mathscinet.ams.org/mathscinet/article?mr=1319345}{1319345}.


\bibitem{elli} Joseph Frederick Elliott, ``\textit{The Characteristic Roots of Certain Real Symmetric Matrices}." Master's Thesis, University of Tennessee, 1953. \href{https://trace.tennessee.edu/cgi/viewcontent.cgi?article=3834&context=utk_gradthes }{URL}. 


\bibitem{gun-ros} R.~C. Gunning and H.~E. Rossi, {\it Analytic functions of several complex variables}, Prentice-Hall, Inc., Englewood Cliffs, NJ, 1965; MR\href{https://mathscinet.ams.org/mathscinet/article?mr=180696}{0180696}

\bibitem{haz-por} S. Hazra and E. Porten, \textit{Truncated tube domains with multi-sheeted envelope}, Proc. Amer. Math. Soc. {\bf 153} (2025), no.~7, 2981--2994; MR\href{https://mathscinet.ams.org/mathscinet/article?mr=4912131}{4912131}.


\bibitem{ivas}  S.M. Iva\v skovi\v c, \textit{Envelopes of holomorphy of some tube sets in $\C^2$ and the monodromy theorem}, Math. USSR, Izv. \textbf{19}, 189--196 (1982).

		
	
\bibitem{JP1} M. Jarnicki and P. Pflug, {\it Extension of holomorphic functions}, second extended edition [of 1797263], De Gruyter Expositions in Mathematics, 34, De Gruyter, Berlin, [2020] \copyright 2020; MR\href{https://mathscinet.ams.org/mathscinet/article?mr=4201928}{4201928}.
	
\bibitem{JP2} M. Jarnicki and P. Pflug, \textit{The envelope of holomorphy of a classical truncated tube domain}, Proc. Amer. Math. Soc. {\bf 150} (2022), no.~2, 687--689; MR\href{https://mathscinet.ams.org/mathscinet/article?mr=4356178}{4356178}
	
\bibitem{jor2} B. J\"oricke, Some remarks concerning holomorphically convex hulls and envelopes of holomorphy, Math. Z. {\bf 218} (1995), no.~1, 143--157; MR\href{https://mathscinet.ams.org/mathscinet/article?mr=1312583}{1312583}.


\bibitem{kra} S.~G. Krantz, {\it Function theory of several complex variables}, reprint of the 1992 edition, 
AMS Chelsea Publ., Providence, RI, 2001; MR\href{https://mathscinet.ams.org/mathscinet/article?mr=1846625}{1846625}.


\bibitem{thie} C. Laurent-Thi\'ebaut, Sur l'extension des fonctions CR dans une vari\'et\'e{} de Stein, Ann. Mat. Pura Appl. (4) {\bf 150} (1988), 141--151; MR\href{https://mathscinet.ams.org/mathscinet/article?mr=946033}{0946033}.

\bibitem{por-thie} C. Laurent-Thi\'ebaut and E. Porten, Analytic extension from non-pseudoconvex boundaries and $A(D)$-convexity, Ann. Inst. Fourier (Grenoble) {\bf 53} (2003), no.~3, 847--857; MR\href{https://mathscinet.ams.org/mathscinet/article?mr=2008443}{2008443}.

\bibitem{Lupa-86} G. Lupacciolu, \textit{A theorem on holomorphic extension of CR-functions}, Pacific J. Math. {\bf 124} (1986), no.~1, 177--191; MR\href{https://mathscinet.ams.org/mathscinet/article?mr=850675}{0850675}.


\bibitem{milnor} J.~W. Milnor, {\it Morse theory}, Annals of Mathematics Studies, No. 51, Princeton Univ. Press, Princeton, NJ, 1963; MR\href{https://mathscinet.ams.org/mathscinet/article?mr=163331}{0163331}.
	
\bibitem{nogu}	J. Noguchi: \textit{A brief proof of Bochner's tube theorem and a generalized tube}, Preprint 2020
on \href{https://arxiv.org/abs/2007.04597}{arxiv.org/pdf/2007.04597.pdf}.


\bibitem{egm} E. Porten, \textit{On generalized tube domains over $\C^n$}, Complex Var. Theory Appl. {\bf 50} (2005), no.~1, 1--5; MR\href{https://mathscinet.ams.org/mathscinet/article?mr=2114348}{2114348}.


\bibitem{rob} Peng Robert, ``\textit{The Hahn-Banach separation theorem and other separation results}", \href{https://math.uchicago.edu/~may/REU2014/REUPapers/Peng.pdf}{https://math.uchicago.edu/may/REU2014/REUPapers/Peng.pdf}, (2014).

\bibitem{roman} S. M. Roman, {\it Advanced linear algebra}, third edition, Graduate Texts in Mathematics, 135, Springer, New York, 2008; MR\href{https://mathscinet.ams.org/mathscinet/article?mr=2344656}{2344656}.


\bibitem{ros-st} J.-P. Rosay and E.~L. Stout, Rad\'o's theorem for CR-functions, Proc. Amer. Math. Soc. {\bf 106} (1989), no.~4, 1017--1026; MR\href{https://mathscinet.ams.org/mathscinet/article?mr=964461}{0964461}.

		
\bibitem{slowd}	Z. Slodkowski, Analytic set-valued functions and spectra, Math. Ann. {\bf 256} (1981), no.~3, 363--386; MR\href{https://mathscinet.ams.org/mathscinet/article?mr=626955}{0626955}.


\bibitem{stol2} G. Stolzenberg, \textit{Polynomially convex sets}, Bull. Amer. Math. Soc. {\bf 68} (1962), 382--387; MR\href{https://mathscinet.ams.org/mathscinet/article?mr=151642}{0151642}.

\bibitem{st-81}  E.~L. Stout, Analytic continuation and boundary continuity of functions of several complex variables, Proc. Roy. Soc. Edinburgh Sect. A {\bf 89} (1981), no.~1-2, 63--74; MR\href{https://mathscinet.ams.org/mathscinet/article?mr=628129}{0628129}.
	
\bibitem{stout} E.~L. Stout, {\it Polynomial convexity}, Progress in Mathematics, 261, Birkh\"auser Boston, Boston, MA, 2007; MR\href{https://mathscinet.ams.org/mathscinet/article?mr=2305474}{2305474}.


\vspace{.5cm}
\end{thebibliography}
\end{document}